\documentclass[reqno]{amsart}

\numberwithin{equation}{section}

\usepackage[colorlinks]{hyperref}

\usepackage{bm}
\usepackage{amssymb}
\usepackage{graphicx} 
\usepackage{amscd}
\usepackage{enumerate}
\usepackage{cancel}
\usepackage[font=footnotesize]{caption}
\usepackage{dsfont}
\usepackage[colorinlistoftodos]{todonotes}

\definecolor{darkred}{rgb}{1,0,0} 
\definecolor{darkgreen}{rgb}{0,0.6,0}
\definecolor{darkblue}{rgb}{0,0,0.8}

\hypersetup{colorlinks,
linkcolor=darkblue,
filecolor=darkgreen,
urlcolor=black,
citecolor=darkgreen}

 \newcommand{\N}{\mathds{N}}
 \newcommand{\Z}{\mathds{Z}}
 \newcommand{\R}{\mathds{R}}
 \newcommand{\C}{\mathds{C}}
 \newcommand{\Wu}{W^{u}}
 \newcommand{\Ws}{W^{s}}

 \newcommand{\Eu}{E^{u}}
 \newcommand{\Es}{E^{s}}
 \newcommand{\Kbr}{K_{\mathrm{br}}}
 \newcommand{\Krad}{K_{\mathrm{rad}}}
 \newcommand{\AAA}{\mathcal{A}}
 \newcommand{\BB}{\mathcal{B}}
 \newcommand{\EE}{\mathcal{E}}
 \newcommand{\FF}{\mathcal{F}}
 \newcommand{\UU}{\mathcal{U}}
 \newcommand{\VV}{\mathcal{V}}
 
 \newcommand{\GG}{\mathcal{G}}
 \newcommand{\JJ}{\mathcal{J}}
 \newcommand{\SSS}{\mathcal{S}}
 \newcommand{\HH}{\mathcal{H}}
 \newcommand{\CC}{\mathcal{C}}
 \newcommand{\DD}{\mathcal{D}}
 \newcommand{\RR}{\mathcal{R}}
 \newcommand{\Sp}{\mathrm{Sp}}
 \newcommand{\Symp}{\mathrm{Symp}}
 \newcommand{\Diff}{\mathrm{Diff}}

 \DeclareMathOperator{\diam}{diam}
 \DeclareMathOperator{\interior}{int}
  
 \DeclareMathOperator{\rank}{rank} 
 \DeclareMathOperator*{\Per}{Per}
 \DeclareMathOperator*{\Fix}{Fix}

 \theoremstyle{plain}
 \newtheorem{MainThm}{Theorem}
 
 \newtheorem{Thm}{Theorem}[section]
 \newtheorem{Prop}[Thm]{Proposition}
 \newtheorem{Lemma}[Thm]{Lemma}
 \newtheorem{Cor}[Thm]{Corollary}
 \theoremstyle{definition}
 
 \newtheorem{Remark}[Thm]{Remark}

\title[Proof of the $C^2$-stability conjecture]{Proof of the $\bm{C^2}$-stability conjecture\\for geodesic flows of closed surfaces} 

\author{Gonzalo Contreras}
\address{Gonzalo Contreras\newline\indent 
Centro de Investigaci\'on en Matem\'aticas\newline\indent 
A.P. 402, 36.000, Guanajuato, GTO, Mexico}
\email{gonzalo@cimat.mx}

\author{Marco Mazzucchelli}
\address{Marco Mazzucchelli\newline\indent CNRS, UMPA, \'Ecole Normale Sup\'erieure de Lyon\newline\indent 46 all\'ee d'Italie, 69364 Lyon, France}
\email{marco.mazzucchelli@ens-lyon.fr}

\thanks{Gonzalo Contreras is partially supported by CONACYT, Mexico, grant A1-S-10145. Marco Mazzucchelli is partially supported by the ANR grants CoSyDy, ANRCE40-0014, and COSY, ANR-21-CE40-0002.}

\date{September 22, 2021}



\begin{document}

\begin{abstract}
We prove that a $C^2$-generic Riemannian metric on a closed surface has either an elliptic closed geodesic or an Anosov  geodesic flow. As a consequence, we prove the $C^2$-stability conjecture for Riemannian geodesic flows of closed surfaces: a $C^2$-structurally stable Riemannian geodesic flow of a closed surface is Anosov. In order to prove these statements, we establish a general result that may be of independent interest and provides sufficient conditions for a Reeb flow of a closed 3-manifold to be Anosov.
\tableofcontents
\end{abstract}

\maketitle

\vspace{-50pt}

\section{Introduction}

\subsection{Main results}
In his seminal paper \cite{Poincare:1905ub}, Poincar\'e claimed that any convex 2-sphere in $\R^3$ has a (possibly degenerate) elliptic simple closed geodesic. As it turned out, this claim was not correct: Grjuntal \cite{Grjuntal:1979uy} provided an example of convex 2-sphere all of whose simple closed geodesics are hyperbolic (the example may have non-hyperbolic closed geodesics with self-intersections). Nevertheless, after dropping the requirement of the absence of self-intersections, Poincar\'e's claim was at least almost correct. Indeed, about a century later, in the year 2000, Hermann announced the existence of an elliptic closed geodesic on any $C^2$-generic positively curved Riemannian 2-sphere. The first author and Oliveira \cite{Contreras:2004vh} extended this result beyond the class of convex 2-spheres: there is an open dense subset $\UU$ of the space of Riemannian metrics on $S^2$ such that, for any choice of Riemannian metric $g\in\UU$, the Riemannian 2-sphere $(S^2,g)$ has an elliptic closed geodesic. The purpose of this paper is to extend this result to all closed surfaces, and derive as a consequence a proof of the $C^2$-stability conjecture for their geodesic flows. The extension will involve the notion of Anosov geodesic flow  \cite{Anosov:1967wm}. We recall that a flow on a closed manifold is called \emph{Anosov} when the whole manifold is hyperbolic for the flow.

Throughout this paper, all Riemannian metrics are assumed to be smooth, meaning $C^\infty$. Nevertheless, we will often endow the space of smooth Riemannian metrics on a given closed manifold with the $C^2$ topology.
Our first main theorem is the following.

\begin{MainThm}
\label{t:main}
For any closed surface $M$, there exists a $C^2$-open dense subset $\UU$ of the space of smooth Riemannian metrics on $M$ such that, for all $g\in\mathcal{U}$, one of the following two conditions is verified:
\begin{itemize}

\item The Riemannian surface $(M,g)$ has an  elliptic closed geodesic.

\item The geodesic flow of $(M,g)$ is Anosov.

\end{itemize}
\end{MainThm}

On closed surfaces of genus at least two, there are non-empty open sets of Riemannian metrics that do not admit any elliptic closed geodesic. It is the case, for instance, for the Riemannian metrics of negative curvature, which are Anosov. It is well known that closed Riemannian surfaces of genus at most one do not have Anosov geodesic flows. Therefore, for surfaces of lower genus, Theorem~\ref{t:main} reduces to the following statement, which generalizes the above mentioned \cite{Contreras:2004vh}.

\begin{Cor}
\label{c:genus_0_1}
For any closed surface $M$ of genus at most one, there exists a $C^2$-open dense subset $\mathcal{U}$ of the space of smooth Riemannian metrics such that, for all $g\in\mathcal{U}$, the Riemannian surface $(M,g)$ has an elliptic closed geodesic.
\end{Cor}

As we anticipated, Theorem~\ref{t:main} has a remarkable application concerning structural stability, which is an utmost important concept in dynamical systems. Let us recall its general notion for discrete dynamical systems. Let $\DD$ be a family of diffeomorphisms of a closed manifold $N$, endowed with a topology. A diffeomorphism $\phi\in\DD$ is \emph{$\DD$-structurally stable} when it has an open neighborhood $\VV\subset\DD$ and, for each diffeomorphism $\psi\in\VV$, there exists a homeomorphism $\kappa:N\to N$ conjugating $\phi$ and $\psi$, i.e.~$\phi=\kappa^{-1}\circ\psi\circ\kappa$. 
If $\DD$ is the space of all smooth diffeomorphisms of $N$, endowed with the $C^k$-topology, this notion is usually called  $C^k$-structural stability.
The celebrated stability conjecture of Palis and Smale \cite{Palis:1970tw} claims that, for $k\geq1$, $C^k$-structurally stable diffeomorphisms of a closed manifold are Axiom~A, meaning that their non-wandering set is hyperbolic and coincides with the closure of the  space of periodic points. For $k=1$, this conjecture was proved by Ma\~n\'e~\cite{Mane:1988vo}.

The analogous notion of structural stability  for geodesic flows of closed Riemannian manifolds is the following. Let $M$ be a closed manifold, and $\GG^k(M)$ be the space of smooth Riemannian metrics on $M$, endowed with the $C^k$ topology. Any Riemannian metric $g$ on $M$ defines a unit tangent bundle $S^gM=\big\{ (x,v)\in TM\ \big|\ \|v\|_g=1 \big\}$ and a geodesic flow $\phi^{g}_t:S^gM\to S^gM$, whose orbits have the form $\phi^g_t(x(0),\dot x(0))=(x(t),\dot x(t))$, where $x:\R\to M$ is a geodesic of $(M,g)$ parametrized with unit speed. The geodesic flow of a Riemannian metric $g_0$ on $M$ is \emph{$C^k$-structurally stable} when $g_0$ has an open neighborhood $\VV\subset\GG^k(M)$ such that, for each $g_1\in\VV$, there exists a homeomorphism $\kappa:S^{g_0}M\to S^{g_1}M$ mapping orbits of $\phi_t^{g_0}$ to orbits of $\phi_t^{g_1}$. We stress that $\kappa$ is not necessarily a conjugacy for the geodesic flows $\phi_t^{g_0}$ and $\phi_t^{g_1}$, since it does not preserve the time-parametrization of the orbits.

Any unit tangent bundle has a canonical volume form, which is preserved by the geodesic flow. Therefore, by Poincare's recurrence theorem, a geodesic flow is Axiom A if and only if it is Anosov.
As a consequence of Theorem~\ref{t:main}, we establish the $C^2$-stability conjecture for Riemannian geodesic flows of closed surfaces.

\begin{MainThm}
\label{t:structural_stability}
The $C^2$-structurally stable geodesic flows of closed Riemannian surfaces are Anosov.
\end{MainThm}

Before this work, partial results on the stability conjecture for  Riemannian geodesic flows were recently obtained, with different techniques than ours, by Ruggiero and Rifford \cite{Rifford:2021aa, Ruggiero:2021aa}. More precisely, they proved the conjecture under the assumption that the closed Riemannian manifold is without conjugate points, and satisfies one of the following extra assumptions: it is a surface, or a 3-dimensional manifold whose universal cover is quasi-convex and has diverging geodesic rays, or any visibility manifold. For closed Riemannian surfaces without conjugate points, the conjecture also follows from the recent work of Climenhaga, Knieper, and War \cite{Climenhaga:2021aa}. Concerning the  assumption of absence of conjugate points in these works, it is worthwhile to mention a classical result of Ruggiero \cite{Ruggiero:1991aa}: on any closed manifold, the space of Riemannian metrics with Anosov geodesic flows is precisely the interior in the $C^2$ topology of the space of Riemannian metrics without conjugate points.

Another important consequence of Theorem~\ref{t:main}, which for the 2-sphere was already stated in \cite[page~1396]{Contreras:2004vh} is the following negative result.

\begin{MainThm}
\label{t:low_genus}
On a closed surface of genus at most one, there is no Riemannian metric whose geodesic flow is $C^2$-structurally stable.
\end{MainThm}

In Theorems~\ref{t:main}, \ref{t:structural_stability}, and \ref{t:low_genus}, it is crucial that we work within the class of Riemannian metrics, and not on the wider class of Finsler metrics. Indeed, for Finsler metrics, all these theorems follow from the celebrated result of Newhouse \cite{Newhouse:1977wn}, which applies to even more general Hamiltonian systems. 

\begin{Thm}[Newhouse]
On a symplectic manifold $(W,\omega)$, any Hamiltonian $H_0:W\to\R$ with a compact level set $H_0^{-1}(1)$ has a $C^2$-open neighborhood $\UU$ and a dense open subset $\VV\subset\UU$ such that any Hamiltonian $H\in\VV$ has either an Anosov Hamiltonian flow on $H^{-1}(1)$ or a  closed orbit on $H^{-1}(1)$ with exactly two Floquet multipliers on the unit circle $S^1\subset\C$. 
\hfill\qed
\end{Thm}

We stress that Newhouse's theorem does not imply   our Riemannian ones. A Riemannian metric $g$ on a closed manifold $M$ can be treated as a Hamiltonian $H_g:TM\to\R$, $H_g(x,v)=\|v\|_g^2$. The geodesic flow $\phi^g_t$ is the Hamiltonian flow of $H_g$ on the level set $H_g^{-1}(1)=S^gM$. However, a Hamiltonian $H:TM\to\R$ that is $C^2$-close to $H_g$ is not necessarily the square of a Riemannian norm; in general, $H$ is only the square of a Finsler norm. In other words, working with general Hamiltonians $H$, one is allowed to do local perturbations of the level set $H^{-1}(1)$, whereas if one only works within the class of Hamiltonians of the form $H_g$, a perturbation of the Riemannian metric $g$ near a point $x\in M$ corresponds to a perturbation of $H_g^{-1}(1)$ near the whole fiber $H_g^{-1}(1)\cap T_xM$. 

A crucial ingredient of Newhouse's theorem is the $C^2$-closing lemma for Hamiltonians \cite{Pugh:1983aa, Arnaud:1998aa}: on any symplectic manifold, a $C^2$ generic Hamiltonian $H$ has subset of closed orbits of its Hamiltonian flow $\phi_H^t$ that is dense in the subset of non-wandering points of $\phi_H^t$.
The analogous statement in the Riemannian category is a hard open problem: it is not known whether, for a $C^2$ generic Riemannian metric on a closed manifold of dimension at least two, the closed orbits of the geodesic flow form a dense subset of the unit tangent bundle. At best, a result of Rifford \cite{Rifford:2012aa} asserts that any given orbit of a Riemannian geodesic flow can be closed after a $C^1$-perturbation of the Riemannian metric; however, a $C^1$ perturbation of a Riemannian metric corresponds to a rough $C^0$ perturbation of its geodesic vector field. In higher topology, a result of Irie \cite{Irie:2015aa} implies that a $C^\infty$ generic Riemannian metric on a closed surface $M$ has closed geodesics that form a dense subset of $M$; while this result remarkably involves the $C^\infty$ topology, it is still very far from the closing lemma, which would require the lifts of the closed geodesics to form a dense subset of the unit tangent bundle of $M$. Our approach in this paper follows the one in Contreras and Oliveira's \cite{Contreras:2004vh}: we avoid the closing lemma altogether, and instead combine tools from contact topology, hyperbolic dynamics, and the Brouwer plane translation theorem.

\subsection{Anosov Reeb flows}
In order to prove Theorem~\ref{t:main}, we establish sufficient conditions for a Reeb flow of a closed contact 3-manifold to be Anosov, a statement that may be of independent interest. We recall that, on a 3-dimensional manifold $N$ which for us will always be closed, a contact form $\lambda$ is a 1-form such that $\lambda\wedge d\lambda$ is nowhere vanishing. Associated with a contact form $\lambda$, there is a vector field $X_\lambda$, called the Reeb vector field, uniquely defined by the equations $\lambda(X_\lambda)\equiv1$ and $d\lambda(X_\lambda,\cdot)\equiv0$. The flow $\phi_t:N\to N$ of $X_\lambda$ is called Reeb flow, and its orbits are referred to as Reeb orbits. Unit tangent bundles and their geodesic flows are a particular example of contact manifolds and Reeb flows. We denote by $\Per(X_\lambda)\subset N$ the subspace of those points lying on closed Reeb orbits.
As usual, for a hyperbolic closed Reeb orbit $\gamma\subset\Per(X_\lambda)$, we denote by $\Ws(\gamma)$ and $\Wu(\gamma)$ its stable and unstable manifolds respectively. Our theorem is the following.

\begin{MainThm}
\label{t:Reeb}
Let $(N,\lambda)$ be a closed contact 3-manifold satisfying the following two properties:
\begin{itemize}

\item[$(i)$] The closure $\overline{\Per(X_\lambda)}$ is uniformly hyperbolic.\vspace{3pt}

\item[$(ii)$] The Kupka-Smale transversality condition holds: $\Ws(\gamma_1)\pitchfork\Wu(\gamma_2)$ for all $($not necessarily distinct$)$ closed Reeb orbits $\gamma_1,\gamma_2\subset\Per(X_\lambda)$.

\end{itemize}
Then the Reeb flow of $(N,\lambda)$ is Anosov.
\end{MainThm}

The arguments in \cite{Contreras:2004vh} relied on Hofer-Wysocki-Zehnder's theory of finite energy foliations \cite{Hofer:1998vy, Hofer:2003wf}, which in turn is based on the theory of pseudo-holomorphic curves from symplectic geometry \cite{Gromov:1985ww, Hofer:2002vt}. The relevant property of finite energy foliations for studying the problems described above is the fact that each of their leaves is a surface of section for the geodesic flow, although not necessarily a section with a well defined first return map. 

Finite energy foliations are only available for Reeb flows of spheres and real projective spaces in dimension 3. For our Theorem~\ref{t:Reeb}, we instead employ the foliations provided by the so-called \emph{broken book decompositions}, recently introduced by Colin, Dehornoy, and Rechtman \cite{Colin:2020tl}, which are available for general non-degenerate Reeb flows on closed contact 3-manifolds.

After this work was completed, we learnt that Schulz very recently proved a partial version of our Theorem~\ref{t:main} in his Ph.D. thesis \cite{Schulz:2021wk}. His result establishes Theorem~\ref{t:main} under a technical extra assumption: the existence of finitely many simple closed geodesics $\gamma_1,...,\gamma_n$ such that every $\gamma_i$ intersects another $\gamma_j$, the complement $M\setminus(\gamma_1\cup...\cup\gamma_n)$ is simply connected, and any other contractible closed geodesic intersects at least one $\gamma_i$. These assumptions guarantee the existence of a suitable system of surfaces of section of annulus type for the geodesic flow.

\subsection{Organization of the paper}
In Section~\ref{s:sos}, after introducing some background on Reeb dynamics and broken book decompositions, we prove two technical statements (Propositions~\ref{p:large_return} and~\ref{p:unbounded_Ws}) concerning the diameter of the images of suitable paths under the arrival map between two pages of a broken book. In Section~\ref{s:hyperbolic}, we briefly recall the needed background from hyperbolic dynamics, and we prove a technical statement (Lemma~\ref{l:Ws_Wu_closed}) concerning heteroclinic rectangles of suitable Reeb flows. Section~\ref{s:proofs} is devoted to the proofs of our main theorems: in Section~\ref{ss:proof_Reeb} we prove Theorem~\ref{t:Reeb}; in Section~\ref{ss:proof_geodesic_flows} we prove Theorem~\ref{t:main}, Corollary~\ref{c:genus_0_1}, and Theorems~\ref{t:structural_stability} and~\ref{t:low_genus}. 

\subsection{Acknowledgements}
Marco Mazzucchelli is grateful to Pierre Dehornoy for a few discussions concerning broken book decompositions. Both authors are grateful to the anonymous referees for their careful reading of the manuscript, and their helpful reports.

\section{Surfaces of section in contact 3-manifolds}\label{s:sos}

\subsection{Surfaces of section}
Let $(N,\lambda)$ be a closed connected contact 3-manifold, with associated Reeb vector field $X=X_\lambda$ and Reeb flow $\phi_t:N\to N$. We recall that the Reeb flow preserves the contact form, i.e.\ $\phi_t^*\lambda=\lambda$.
A \emph{surface of section} for the Reeb vector field $X$ is a compact, orientable, immersed surface $\Sigma\looparrowright N$ satisfying the following two properties:
\begin{itemize}
\item (\emph{Boundary}) Every connected component of the boundary $\partial\Sigma$ is the covering map of a closed Reeb orbit, that is, a closed orbit of $X$.

\item (\emph{Transversality}) The interior $\interior(\Sigma)$ is embedded in $N\setminus\partial\Sigma$ and transverse to the Reeb vector field~$X$.
\end{itemize}
The transversality assumption implies that $d\lambda$ restricts to an area form on $\interior(\Sigma)$, and we orient $\Sigma$ so that such an area form is positive. By Stokes' theorem, $\Sigma$ must have non-empty boundary.
We remark that, in the literature, the notion of surface of section is sometimes different  than the one employed here: for instance, $\Sigma$ may be required to be embedded in $N$, and the induced orientation on the boundary $\partial\Sigma$ may be required to coincide with the orientation given by the Reeb vector field $X|_{\partial\Sigma}$.

Following the classical approach pioneered by Poincar\'e and Birkhoff, surfaces of sections can be employed in order to reduce the study of a 3-dimensional flow to the study of a surface diffeomorphism; this requires one further condition, beyond the already mentioned transversality and boundary properties. A surface of section $\Sigma$ is called a \emph{Birkhoff section} when it satisfies the following extra property:
\begin{itemize}
\item (\emph{Globality}) Every Reeb orbit intersects $\Sigma$ in uniformly bounded positive and negative time, i.e.\ there exists $T>0$ such that, for every $z\in N$, there exist $t_1\in[-T,0)$ and $t_2\in(0,T]$ such that $\phi_{t_1}(z)\in\Sigma$ and $\phi_{t_2}(z)\in\Sigma$.
\end{itemize}
The first return time to a Birkhoff section $\Sigma$ is the smooth function
\begin{align*}
\tau:\interior(\Sigma)\to(0,\infty),\qquad\tau(z)=\min\big\{t>0\ \big|\ \phi_t(z)\in\Sigma\big\},
\end{align*}
and the first return map (also called the Poincar\'e map) is the diffeomorphism
\begin{align*}
 \psi:\interior(\Sigma)\to\interior(\Sigma),\qquad \psi(z)=\phi_{\tau(z)}(z).
\end{align*}
Such a diffeomorphism is area preserving, i.e.\ $\psi^*d\lambda=d\lambda$.

\subsection{Non-degenerate closed Reeb orbits}
\label{ss:non_degenerate}
Let $\gamma(t)=\phi_t(z)$ be a closed Reeb orbit of minimal period $t_0>0$. The restriction $d\phi_{t_0}(z)|_{\ker(\lambda_z)}$ is a linear symplectomorphism with respect to the symplectic form $d\lambda_z|_{\ker(\lambda_z)}$. Its eigenvalues, which must be of the form 
$\sigma,\sigma^{-1}\in S^1\cup\R\setminus\{0\}$, are called the Floquet multipliers of $\gamma$. Here, $S^1$ denotes the unit circle in the complex plane $\C$. The closed Reeb orbit $\gamma$ is called:
\begin{itemize}
 \item \emph{elliptic} when $\sigma\in S^1$;
 \item \emph{positively hyperbolic} when $\sigma\in (0,1)\cup(1,\infty)$;
 \item \emph{negatively hyperbolic} when $\sigma\in (-\infty,-1)\cup(-1,0)$;
 \item \emph{non-degenerate} when $\sigma\neq1$.
\end{itemize}
The contact form $\lambda$ (or its Reeb vector field $X$) is \emph{non-degenerate} when all the closed Reeb orbits are non-degenerate for all possible periods, i.e.
\begin{align}
\label{e:strong_nondegeneracy}
 \ker(d\phi_{t}(z)-I)=\mathrm{span}\{X(z)\},\qquad\forall t>0,\ z\in\mathrm{fix}(\phi_{t}).
\end{align}

Let $\gamma(t)=\phi_t(z)$ be a hyperbolic closed Reeb orbit of minimal period $t_0>0$ and Floquet multipliers $\sigma,\sigma^{-1}\in\R$, with $0<|\sigma|<1$. Along $\gamma$, the tangent bundle of our contact manifold splits as 
$TN|_{\gamma}=\Es\oplus\Eu\oplus\mathrm{span}\{X\}$,
where
\begin{align*}
\Es(z)=\ker(d\phi_{t_0}(z)-\sigma I),\qquad
\Eu(z)=\ker(d\phi_{t_0}(z)-\sigma^{-1} I)
\end{align*}
are called the \emph{stable} and \emph{unstable bundles} of $\gamma$ respectively. These vector bundles are invariant by the linearized Reeb flow, i.e.~
$d\phi_t(z)\Es(z)=\Es(\phi_t(z))$ and $d\phi_t(z)\Eu(z)=\Eu(\phi_t(z))$.
Let $d:N\times N\to[0,\infty)$ be a distance on $N$ induced by an arbitrary Riemannian metric. For each point $z\in\gamma$, the spaces
\begin{align*}
\Ws(z)&=\Big\{ z'\in N\ \Big|\ \lim_{t\to\infty}d(\phi_t(z'),\phi_t(z))=0 \Big\},\\
\Wu(z)&=\Big\{ z'\in N\ \Big|\ \lim_{t\to-\infty}d(\phi_t(z'),\phi_t(z))=0 \Big\}
\end{align*}
turn out to be smooth 1-dimensional injectively immersed submanifolds of $N$ containing the point $z$ in their interior, and with tangent spaces $T_z\Ws(z)=\Es(z)$ and  $T_z\Wu(z)=\Eu(z)$. 
The \emph{stable} and \emph{unstable manifolds} of $\gamma$ are defined respectively as
\begin{align*}
\Ws(\gamma)=\bigcup_{z\in\gamma} \Ws(z),\qquad
\Wu(\gamma)=\bigcup_{z\in\gamma} \Wu(z).
\end{align*}
These are 2-dimensional injectively immersed submanifolds of $N$ invariant by the Reeb flow, with tangent spaces at every $z\in\gamma$ given by
\begin{align*}
T_z \Ws(\gamma) = \Es(z)\oplus\mathrm{span}\big\{X(z)\big\},\qquad
T_z \Wu(\gamma) = \Eu(z)\oplus\mathrm{span}\big\{X(z)\big\}.
\end{align*}
If $\gamma$ is positively hyperbolic (i.e.~$\sigma>0$), then $\Es$ and $\Eu$ are orientable line bundles over $\gamma$; in this case $\Ws(\gamma)$ and $\Wu(\gamma)$ are diffeomorphic to open annuli $A=S^1\times\R$, with a diffeomorphism that sends $\gamma$ to $S^1\times\{0\}\subset A$. If instead $\gamma$ is negatively hyperbolic (i.e.~$\sigma<0$), then $\Es$ and $\Eu$ are non-orientable line bundles over $\gamma$, and $\Ws(\gamma)$ and $\Wu(\gamma)$ are diffeomorphic to open Mobius bands $M=[0,1]\times\R/\sim$, where $(0,y)\sim(1,-y)$, with a diffeomorphism that sends $\gamma$ to $S^1\times\{0\}\subset M$.

\subsection{Open/broken book decompositions}
\label{ss:broken_book_decomposition}
A Birkhoff section $\Sigma$ whose first return map $\tau:\interior(\Sigma)\to(0,\infty)$ extends smoothly to the boundary $\partial\Sigma$ defines a so-called \emph{rational open book decomposition} of the contact manifold $(N,\lambda)$. Namely, for each $s\in[0,1]$, we set
$\Sigma_s:=\big\{\phi_{s\tau(z)}(z)\ \big|\ z\in\Sigma\big\}.$
Notice that $\Sigma=\Sigma_0=\Sigma_1$, and each $\Sigma_s$ is a Birkhoff section with the same boundary $K:=\partial\Sigma_s=\partial\Sigma$. The collection of interiors $\interior(\Sigma_s)$ foliates the complement $N\setminus K$. Suggestively, the $\Sigma_s$ are called the pages and the boundary $K$ is called the binding of the rational open book.

It is not known whether every Reeb vector field on a closed contact 3-manifold admits a Birkhoff section\footnote{A few months after this work was completed, the authors established the existence of Birkhoff sections for all Reeb flows on 3-dimensional closed manifolds whose closed orbits are non-degenerate and whose hyperbolic closed orbits satisfy the Kupka-Smale transversality condition \cite{Contreras:2022aa}; the proof crucially requires, among other ingredients, the results of Section~\ref{ss:size_arrival}. Independently, Colin, Dehornoy, Hryniewicz, and Rechtman \cite{Colin:2022aa} established a related alternative result: the existence of Birkhoff sections for all Reeb flows on 3-dimensional closed manifolds whose closed orbits are non-degenerate and equidistributed, i.e.~the volume form of the contact manifold can be approximated as a measure by a sequence of measures supported on finite collections of closed orbits.} (and thus an associated rational open book decomposition). 
Nevertheless, recent work of Colin, Dehornoy, and Rechtman \cite[Th.~1.1]{Colin:2020tl} implies that any such Reeb vector field admits the following weaker kind of decomposition provided it is non-degenerate. A \emph{broken book decomposition} of the non-degenerate closed contact 3-manifold $(N,\lambda)$ is given by the following data:
\begin{itemize}
\item The \emph{radial binding} $\Krad\subset N$, which is the union of finitely many closed Reeb orbits, each one allowed to be either elliptic or hyperbolic.

\item The \emph{broken binding} $\Kbr\subset N\setminus\Krad$, which is the union of finitely many hyperbolic closed Reeb orbits.

\item A family $\FF$ of closed surfaces with boundary immersed in $N$, called the \emph{pages} of the broken book.

\end{itemize}
The disjoint union $K:=\Krad\cup\Kbr$ is called the \emph{binding} of the broken book. These data satisfy the following properties. 
\begin{itemize}

\item (\emph{Pages}) For each page $\Sigma\in\FF$, the interior $\interior(\Sigma)$ is properly embedded in $N\setminus K$ and transverse to the Reeb vector field $X$, while each connected component of the boundary $\partial\Sigma$ is a covering map of a closed Reeb orbit in the binding $K$.\vspace{5pt}

\item (\emph{Foliation}) The family of interiors $\interior(\FF):=\{\interior(\Sigma)\ |\ \Sigma\in\FF\}$ is a cooriented foliation of the complement of the binding $N\setminus K$.\vspace{5pt}

\item (\emph{Binding}) Near each binding closed Reeb orbit $\gamma\subset K$, let $D\subset N$ be an embedded open two-dimensional disk intersecting $\gamma$ transversely in a single point $z$, and small enough so that it also intersects transversely the interior of each page $\interior(\Sigma)\in\interior(\FF)$ with $\gamma\subset\partial\Sigma$. Consider the restriction  
\[\FF|_D=\{\Sigma\cap D\ |\ \Sigma\in\FF\}.\]
If $\gamma\subset\Krad$, the connected components of the leaves of the singular foliation $\FF|_D$ are radial as in Figure~\ref{f:broken}(a). If instead $\gamma\in\Kbr$, they are radial in four sectors, and hyperbolic in the four sectors separating the previous ones, as in Figure~\ref{f:broken}(b).\vspace{5pt}

\item (\emph{Return time near the radial binding}) There exist an open neighborhood $V\subset N$ of the radial binding and a finite constant $T>0$ such that, for each page $\Sigma\in\FF$ and each point $z\in\Sigma\cap V$, we have $\phi_t(z)\in\Sigma$ for some $t\in(0,T)$.\vspace{5pt}

\item (\emph{Rigid pages}) There exist finitely many pages $\Sigma_1,...,\Sigma_h\in\FF$ such that every Reeb orbit intersects at least one of these pages. If a Reeb orbit $t\mapsto\phi_t(z)$ does not intersect $\interior(\Sigma_1)\cup...\cup\interior(\Sigma_h)$ for arbitrarily large positive times $t>0$, then the $\omega$-limit of such Reeb orbit is some broken boundary component $\gamma_+\in\Kbr$. Analogously, if $t\mapsto\phi_t(z)$ does not intersect $\interior(\Sigma_1)\cup...\cup\interior(\Sigma_h)$ for arbitrarily large negative times $t<0$, then the $\alpha$-limit of such Reeb orbit is some broken boundary component $\gamma_-\in\Kbr$.

\end{itemize}

\begin{figure}
\begin{footnotesize}
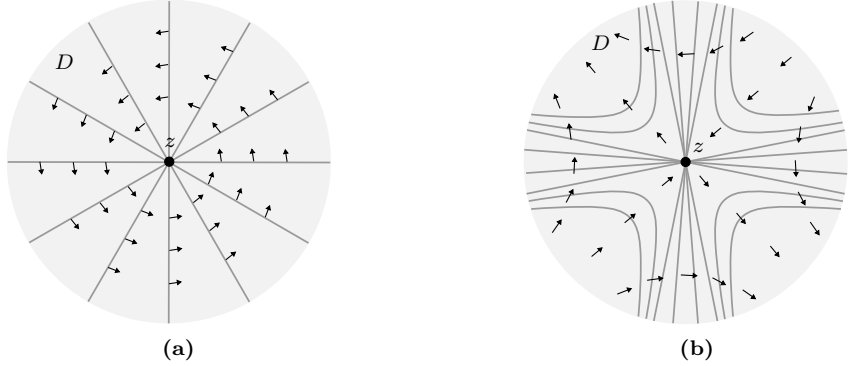
\end{footnotesize}
\caption{\textbf{(a)}~Foliation $\FF|_D$ around a point $z$ on a radial binding component. \textbf{(b)}~Foliation $\FF|_D$ around a point $z$ on a broken binding component. In both pictures, the arrows are positively tangent to a projection of the Reeb vector field $X$ to $D$.}
\label{f:broken}
\end{figure}

Notice that any page $\Sigma\in\FF$ is a surface of section, although not necessarily a Birkhoff section. Nevertheless, we can still define the return map to a given page, except that such map may not be defined on the whole page. More generally, given two (not necessarily distinct) pages $\Sigma_0,\Sigma_1\in\FF$ and a point $z_0\in\interior(\Sigma_0)$ such that $\phi_{t_0}(z_0)\in\Sigma_1$ for some $t_0\in\R\setminus\{0\}$, by the implicit function theorem there exists a maximal open neighborhood $U\subset\interior(\Sigma_0)$ of $z_0$ and a smooth function
$\tau:U\to\R\setminus\{0\}$
such that $\tau(z_0)=t_0$ and $\phi_{\tau(z)}(z)\in\Sigma_1$ for all $z\in U$. The smooth map
$\psi:U\to\Sigma_1$, $\psi(z)=\phi_{\tau(z)}(z)$
is a diffeomorphism onto its image that preserves the area form $d\lambda$, i.e.~$\psi^*(d\lambda|_{\Sigma_1})=d\lambda|_U$.

\subsection{Size of images of the arrival map}
\label{ss:size_arrival}

The following lemmas are among the main ingredients in the proof of Theorem~\ref{t:Reeb}. The arguments are reminiscent of those in Contreras-Oliveira \cite[Sect.~2.4]{Contreras:2004vh}.

\begin{Prop}
\label{p:unbounded_Ws}
Let $(N,\lambda)$ be a closed contact 3-manifold with a non-degenerate Reeb vector field, endowed with a broken book decomposition with binding $K=\Krad\cup\Kbr$. We consider two pages $\Sigma_0,\Sigma_1\subset N$, and the stable manifold
\begin{align*}
\Ws(\Kbr)=\bigcup_{z\in \Kbr} \Ws(z).
\end{align*}
Let $\zeta:[0,1]\to\interior(\Sigma_0)$ be a smooth path, and $\tau:[0,1)\to(0,\infty)$ a smooth unbounded function such that $\phi_{\tau(s)}(\zeta(s))\in\interior(\Sigma_1)$ for all $s\in[0,1)$. Then $\tau(s)\to+\infty$ as $s\to1$, and $\zeta(1)\in\Ws(\Kbr)$.
\end{Prop}

\begin{Prop}
\label{p:large_return}
Let $(N,\lambda)$ be a closed contact 3-manifold with a non-degenerate Reeb vector field, endowed with a broken book decomposition with binding $K=\Krad\cup\Kbr$, and such that $\Ws(\gamma_1)\pitchfork\Wu(\gamma_2)$ for all broken binding components $\gamma_1,\gamma_2\subset \Kbr$.
For all pages $\Sigma_0,\Sigma_1\subset N$ of the broken book, there exists a constant
\[c=c(\Sigma_0,\Sigma_1)>0\] 
with the following property. Let $\zeta_0:[0,1]\to\interior(\Sigma_0)$ be any smooth path such that:
\begin{itemize}

\item[$(i)$] $\zeta_0(1)$ is a transverse intersection of $\zeta_0$ with the stable manifold $\Ws(\Kbr)$,

\item[$(ii)$] there exists a smooth unbounded function $\tau:[0,1)\to(0,\infty)$ such that
 \[\zeta_1(s):=\phi_{\tau(s)}(\zeta_0(s))\in\interior(\Sigma_1),\qquad\forall s\in[0,1).\]

\end{itemize}
Then
\[\diam\big(\zeta_1([0,1))\big)> c.\] 
Indeed, $\zeta_1$ accumulates on a smooth embedded circle $S\subset\Wu(\Kbr)\cap\Sigma_1$, i.e.
\[S\subset\overline{\zeta_1([0,1))}.\]
\end{Prop}

In this proposition, we have denoted by ``$\diam$'' the diameter measured with respect to an arbitrary Riemannian distance $d:N\times N\to[0,\infty)$, i.e.
\begin{align*}
\diam(V)=\sup\big\{ d(z,z')\ \big|\ z,z'\in V \big\},\qquad\forall V\subseteq N.
\end{align*}
Clearly, since $N$ is compact, the conclusion of the proposition is independent of the choice of the Riemannian distance.

\begin{Remark}\label{r:lambda_-lambda}
Replacing the contact form $\lambda$ by its negative $-\lambda$, and thus the Reeb vector field $X$ with $-X$, we infer that Propositions~\ref{p:unbounded_Ws} and~\ref{p:large_return} also hold if we replace the stable manifolds by the corresponding unstable manifolds, and the unbounded function $\tau:[0,1)\to(0,\infty)$ by an unbounded function $\tau:[0,1)\to(-\infty,0)$. The first assertion of Proposition~\ref{p:unbounded_Ws} then becomes $\tau(s)\to-\infty$ as $s\to1$.
\hfill\qed
\end{Remark}

The proofs of Propositions~\ref{p:unbounded_Ws} and~\ref{p:large_return} will require some preliminary lemmas.
Consider a closed contact 3-manifold $(N,\lambda)$ whose Reeb vector field $X$ is non-degenerate, endowed with a broken book decomposition with binding $K=\Krad\cup\Kbr\subset N$. By a \emph{compact 1-parameter family of pages} we mean a family of pages $\Sigma_r\subset N$, indexed by $r\in[r_1,r_2]\subset\R$, such that there exist a point $z_0\in\interior(\Sigma_{r_1})$ and a diffeomorphism onto its image $T_{z_0}:[r_1,r_2]\to[0,\infty)$ such that $T_{z_0}(r_1)=0$ and $\phi_{T_{z_0}(r)}(z_0)\in\Sigma_r$ for all $r\in[r_1,r_2]$. We do not require that $\Sigma_r\neq\Sigma_{r'}$ if $r\neq r'$, although this automatically holds if $r$ and $r'$ are sufficiently close. For each $z\in\interior(\Sigma_{r_1})$, there exist a maximal $r_z\in[r_1,r_2]$ and a unique monotone increasing diffeomorphism onto its image \[T_z:[r_1,r_z)\to[0,\infty)\] such that $T_z(r_1)=0$ and $\phi_{T_z(r)}(z)\in\Sigma_r$ for all $r\in[r_1,r_z)$. 
Here, $T_z$ is not defined if $r_z=r_1$.
The function
\begin{align*}
 (r,z)\mapsto T(z,r)=T_{z}(r)
\end{align*}
is smooth on its domain.

\begin{Lemma}
\label{l:unstable_Ws}
Let $z\in\interior(\Sigma_0)$ be such that $r_z>r_1$ and $T_z(r)\to\infty$ as $r\to r_z$. Then $z\in\Ws(\Kbr)$. 
\end{Lemma}

\begin{proof}
Let us assume by contradiction that $r_z>r_1$ and $T_z(r)\to\infty$ as $r\to r_z$, but $z\not\in\Ws(\Kbr)$. This latter condition implies that there exist an open neighborhood $U\subset N$ of the broken binding $\Kbr$ and a sequence $t_n\to\infty$ such that $z_n:=\phi_{t_n}(z)\not\in U$. Let $r_n\in[r_1,r_z)$ be the  parameters such that $T_z(r_n)=t_n$. 

We claim that there exists an open neighborhood $V\subset N$ of the radial binding $\Krad$ such that $z_n\not\in V$ for all $n$. Indeed, since the interval $[r_1,r_2]$ is compact, there is a constant $k>0$ such that, for every page $\Sigma\subset N$ of the broken book, there are at most $k$ values of $r\in[r_1,r_2]$ such that $\Sigma=\Sigma_r$. For every open neighborhood $V'\subset N$ of $\Krad$, there exists a smaller open neighborhood $V\subset V'$ of $\Krad$ such that, for every page $\Sigma\subset N$ of the broken book and every point $z'\in\interior(\Sigma)\cap V$, there exist $k$ positive times $0<s_1<s_2<...<s_{k}$ such that $\phi_{s_i}(z')\in\interior(\Sigma)\cap V'$ for all $i=1,...,k$. This implies that no $z_n$ can enter $V$, for otherwise the function $T_z$ would be uniformly bounded from above.

If $\delta>0$ is sufficiently small, the pages $\Sigma_r$ with $r\in[r_z-\delta,r_z]$ are pairwise distinct. We set
\begin{align*}
 Z_\delta:=\bigcup_{r\in[r_z-\delta,r_z]} \Sigma_r.
\end{align*}
Up to further reducing $\delta>0$, there exists $t'>0$ such that 
\begin{align*}
 \phi_{t'}(z')\not\in Z_\delta,\qquad\forall z'\in Z_\delta\setminus(U\cup V).
\end{align*}
We already showed that the sequence $z_n$ belongs to $N\setminus(U\cup V)$. If $n_0$ is large enough, we have $r_{n_0}\in [r_z-\delta,r_z]$, that is, $z_{n_0}\in Z_\delta\setminus(U\cup V)$. This implies that 
\[t_n = t_n-T_z(r_{n_0}) + T_z(r_{n_0})<t'+T_z(r_{n_0}),
\qquad \forall n\geq n_0,\] 
which contradicts the unboundedness of the sequence $t_n$.
\end{proof}

\begin{proof}[Proof of Proposition~\ref{p:unbounded_Ws}]
We consider the compact 1-parameter family of pages $\Sigma_r\subset N$, $r\in[0,1]$, uniquely defined by $\phi_{r\tau(0)}(\zeta(0))\in\interior(\Sigma_r)$. We employ the notation introduced just before Lemma~\ref{l:unstable_Ws}: for each $z=\zeta(s)$, we consider the maximal $r_z\in[0,1]$ for which we have a monotone increasing diffeomorphism onto its image 
\[[0,r_z)\to[0,\infty),\qquad r\mapsto T_z(r)=T(z,r),\] 
such that $T(z,0)=0$ and $\phi_{T(z,r)}(z)\in\Sigma_r$ for all $r\in[0,r_z)$. Notice that
\begin{align*}
r_{z}=1,\quad
\lim_{r\to1} T(z,r)=\tau(s),\qquad\forall s\in[0,1),\ z=\zeta(s).
\end{align*}
Moreover, the point $z_*:=\zeta(1)$ has positive parameter $r_{z_*}>0$.

We claim that
\begin{align*}
 \lim_{r\to r_{z_*}} T(z_*,r)=\infty,
\end{align*}
which, according to Lemma~\ref{l:unstable_Ws}, implies
$z_*\in\Ws(\Kbr)$.
Indeed, let us assume by contradiction that the monotone function $r\mapsto T(z_*,r)$ is bounded from above, 
so that
\begin{align*}
 \lim_{r\to r_{z_*}} T(z_*,r) = t_*<\infty.
\end{align*}
This, together with the fact that $\phi_{T(z_*,r)}(z_*)\in\interior(\Sigma_r)$ for all $r\in[0,r_{z_*})$,  implies that $\phi_{t_*}(z_*)\in\interior(\Sigma_{r_{z_*}})$ and that, for each $\epsilon>0$, there exists $\delta>0$ such that $\tau(s)\in(t_*-\epsilon,t_*+\epsilon)$ for all $s\in[1-\delta,1)$, contradicting the unboundedness of the function $\tau$.

Finally, for each $c>0$, there exists $r_c\in(0,r_{z_*})$ such that $T(z_*,r_c)>c$. For all $s\in[0,1)$ sufficiently close to $1$, we have $\tau(s)>T(\zeta(s),r_c)>c-1$. This proves that $\tau(s)\to\infty$ as $s\to1$.
\end{proof}

\begin{Lemma}
\label{l:diam_circle}
Let $(N,\lambda)$ be a closed contact 3-manifold, $\Sigma\subset N$ a surface of section for its Reeb flow, $\gamma$ a hyperbolic closed Reeb orbit, and $W$ either the stable manifold $\Ws(\gamma)$ or the unstable manifold $\Wu(\gamma)$. Then there is a constant $c=c(\Sigma,\gamma)>0$ such that the image of any embedding $S^1\hookrightarrow\interior(\Sigma)\cap W\setminus\gamma$ has diameter larger than $c$.
\end{Lemma}

\begin{proof}
Let $\iota:S^1\hookrightarrow \interior(\Sigma)\cap W\setminus\gamma$ be an embedding as in the statement, and $\zeta:=\iota(S^1)$ its image. Since $\Sigma$ is a compact surface whose boundary is immersed in $N$ and whose interior $\interior(\Sigma)$ is embedded in $N$, there is a lower bound for the diameter of non-contractible closed curves in $\Sigma$. Therefore, we only need to consider the case in which $\zeta$ is contractible in $\Sigma$.

Notice that $W\setminus\gamma$ is either an open annulus (if $\gamma$ is negatively hyperbolic) or a disjoint union of two open annuli (if $\gamma$ is positively hyperbolic). Since $\zeta$ is contained in $\interior(\Sigma)$,  
it is nowhere tangent to the Reeb vector field $X$. Since $X$ is tangent to $W$, we infer that $\zeta$ is non-contractible in $W$; indeed, if $\zeta$ were contractible in $W\setminus\gamma$, it would bound a disk $D\subset W\setminus\gamma$ tangent to $X$, and Poincar\'e-Bendixon's theorem would imply that $D$ contains a closed orbit of $X$, which is impossible since $\gamma$ is the only closed orbit of $X$ in~$W$. 

Therefore there exists an open annulus $A\subset W\setminus\gamma$ with boundary $\partial A=\gamma\cup\zeta$. One component of $\partial A$ is a $k$-fold covering map of $\gamma$, with $k=2$ if $\gamma$ is negatively hyperbolic, and $k=1$ if $\gamma$ is positively hyperbolic. We orient $A$ so that the orientation induced on $\gamma$ as its boundary agrees with the orientation on $\gamma$ provided by the Reeb vector field $X$. We then orient $\zeta$ as a negatively oriented boundary component of $A$. Since $X$ is tangent to $W$, the 2-form $d\lambda|_W$ vanishes, and Stokes' theorem gives
\begin{align}
\label{e:int_zeta_lambda}
\int_\zeta\lambda = -\int_{A} d\lambda + \int_\gamma \lambda = \int_\gamma \lambda = kp_\gamma \geq p_\gamma,
\end{align}
where $p_\gamma>0$ is the minimal period of $\gamma$.

Since $\zeta$ is contractible in $\Sigma$, it is the oriented boundary of an embedded disk $D\subset\interior(\Sigma)$. Since $\interior(\Sigma)$ is transverse to the Reeb vector field $X$, the 2-form $d\lambda$ restricts to an area form on $\Sigma$. Therefore, by Stokes' theorem and~\eqref{e:int_zeta_lambda}, we have
\begin{align*}
 \int_D |d\lambda| = \left| \int_D d\lambda \right| 
 = \left| \int_\zeta \lambda \right|\geq p_\gamma.
\end{align*}
This uniform lower bound for the area of $D$, together with the fact that $\Sigma$ is compact and immersed in $N$, implies a uniform lower bound for the diameter of the boundary $\zeta=\partial D$.
\end{proof}

\begin{proof}[Proof of Proposition~\ref{p:large_return}]
Let $\zeta_0:[0,1]\to\interior(\Sigma_0)$ be a path satisfying properties (i) and (ii) in the statement. We consider the space
\begin{align*}
V:=\Big\{ \phi_{t}(\zeta_0(s))\ \Big|\ s\in[0,1),\ t\in[0,\tau(s)]\Big\}.
\end{align*}
We consider the compact 1-parameter family of pages $\Sigma_r\subset N$, $r\in[0,1]$, defined by $\phi_{r\tau(0)}(\zeta(0))\in\interior(\Sigma_r)$. By the implicit function theorem, there exists a maximal open subset $U_0\subset \Sigma_0$ containing $\zeta_0([0,1))$, and a unique smooth function 
\[T_0:U_0\times [0,1]\to[0,\infty)\] 
such that $T_0(\cdot,0)\equiv0$ and $\phi_{T_0(z,r)}(z)\in\Sigma_r$ for all $z\in U_0$, $r\in[0,1]$. Notice that $T_0(\zeta_0(s),1)=\tau(s)$ for all $s\in[0,1)$. Moreover,
$r\mapsto T_0(z,r)$ is a monotone increasing diffeomorphism onto its image for all $z\in U_0$. The paths $\zeta_0$ and $\zeta_1$ are interpolated by the family of paths
\begin{align*}
 \zeta_r:[0,1)\to\interior(\Sigma_r),\quad \zeta_r(s):=\phi_{T_0(\zeta_0(s),r)}(\zeta_0(s)),\qquad r\in[0,1].
\end{align*}

Since $\zeta_0(1)\in\Ws(\Kbr)$, there exists a hyperbolic closed Reeb orbit $\gamma_0\subset \Kbr$ such that $\zeta_0(1)\in\Ws(\gamma_0)$.
For each $\epsilon>0$ small enough, we consider the local unstable manifold
\begin{align*}
\Wu_{\epsilon}(\gamma_0) = \bigcup_{z\in\gamma_0} \Wu_\epsilon(z),
\end{align*}
where
\begin{align*}
\Wu_{\epsilon}(z) &= \left\{ z'\in\Wu(z)\ \left|\ \displaystyle\sup_{t\leq0} d(\phi_t(z),\phi_t(z'))\leq\epsilon \right.\right\}. 
\end{align*}
The space $\Wu_{\epsilon}(\gamma_0)$ is a compact embedded submanifold of $N$ with boundary, containing $\gamma_0$ in its interior.
Since $\tau$ is unbounded, Proposition~\ref{p:unbounded_Ws} implies that $\tau(s)\to+\infty$ as $s\to1$. Moreover, we assumed that $\dot\zeta_0(1)$ is transverse to the stable manifold $\Ws(\gamma_0)$. Therefore, the $\lambda$-lemma from hyperbolic dynamics implies that $V$ accumulates on a path-connected component $W_\epsilon\subset \Wu_\epsilon(\gamma_0)\setminus\gamma_0$  for $\epsilon>0$ small enough, i.e.
\begin{align}
\label{l:lambda_lemma}
\overline V\supseteq W_\epsilon.
\end{align}

\begin{figure}
\begin{footnotesize}
\begingroup%
  \makeatletter%
  \providecommand\color[2][]{%
    \errmessage{(Inkscape) Color is used for the text in Inkscape, but the package 'color.sty' is not loaded}%
    \renewcommand\color[2][]{}%
  }%
  \providecommand\transparent[1]{%
    \errmessage{(Inkscape) Transparency is used (non-zero) for the text in Inkscape, but the package 'transparent.sty' is not loaded}%
    \renewcommand\transparent[1]{}%
  }%
  \providecommand\rotatebox[2]{#2}%
  \newcommand*\fsize{\dimexpr\f@size pt\relax}%
  \newcommand*\lineheight[1]{\fontsize{\fsize}{#1\fsize}\selectfont}%
  \ifx\svgwidth\undefined%
    \setlength{\unitlength}{177.28176189bp}%
    \ifx\svgscale\undefined%
      \relax%
    \else%
      \setlength{\unitlength}{\unitlength * \real{\svgscale}}%
    \fi%
  \else%
    \setlength{\unitlength}{\svgwidth}%
  \fi%
  \global\let\svgwidth\undefined%
  \global\let\svgscale\undefined%
  \makeatother%
  \begin{picture}(1,0.6507307)%
    \lineheight{1}%
    \setlength\tabcolsep{0pt}%
    \put(0,0){\includegraphics[width=\unitlength,page=1]{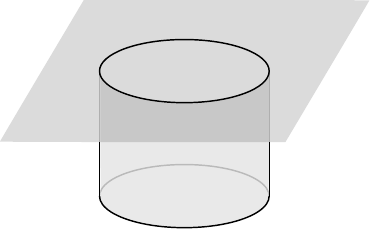}}%
    \put(0.0782614,0.29466777){\color[rgb]{0,0,0}\makebox(0,0)[lt]{\lineheight{1.25}\smash{\begin{tabular}[t]{l}$\Sigma_{r_0}$\end{tabular}}}}%
    \put(0.36593933,0.00699034){\color[rgb]{0,0,0}\makebox(0,0)[lt]{\lineheight{1.25}\smash{\begin{tabular}[t]{l}$\gamma_0$\end{tabular}}}}%
    \put(0.34055598,0.18467364){\color[rgb]{0,0,0}\makebox(0,0)[lt]{\lineheight{1.25}\smash{\begin{tabular}[t]{l}$W_\epsilon$\end{tabular}}}}%
    \put(0.34055598,0.48081207){\color[rgb]{0,0,0}\makebox(0,0)[lt]{\lineheight{1.25}\smash{\begin{tabular}[t]{l}$S_0$\end{tabular}}}}%
    \put(0.69592209,0.5992672){\color[rgb]{0,0,0}\makebox(0,0)[lt]{\lineheight{1.25}\smash{\begin{tabular}[t]{l}$\zeta_{r_0}$\end{tabular}}}}%
    \put(0,0){\includegraphics[width=\unitlength,page=2]{z_S0.pdf}}%
  \end{picture}%
\endgroup%

\end{footnotesize}
\caption{The path $\zeta_{r_0}$ accumulating on  $S_0=\Sigma_{r_0}\cap W_\epsilon$.}
\label{f:S0}
\end{figure}

Up to reducing $\epsilon>0$, there exists a page $\Sigma_{r_0}\subset N$, for some ${r_0}\in(0,1)$, whose interior  intersects $\interior(W_\epsilon)$ transversely in an embedded circle $S_0$, i.e. 
\begin{align*}
\interior(\Sigma_{r_0})\cap\interior(W_\epsilon)=S_0.
\end{align*}
We require $\Sigma_{r_0}$ to be distinct from $\Sigma_1$.
By~\eqref{l:lambda_lemma}, we have $S_0
\subset
\overline V\cap\Sigma_{r_0}$, that is
\begin{align}
\label{e:accumulation}
S_0
\subset
\overline{\zeta_{r_0}([0,1))},
\end{align}
see Figure~\ref{f:S0}. 
By the implicit function theorem, we can find a maximal path-connected open subset $U_1\subset\interior(\Sigma_{r_0})$ containing $\zeta_{r_0}([0,1))$ and a unique smooth function 
\[T_1:U_1\times[r_0,1]\to[0,\infty)\] 
such that $T_1(\cdot,r_0)\equiv0$ and $\phi_{T_1(z,r)}(z)\in\Sigma_r$ for all $r\in[r_0,1]$. The function $r\mapsto T_1(z,r)$ is a monotone increasing diffeomorphism onto its image for all $z\in U_1$, and 
\begin{align*}
 T_1(\zeta_{r_0}(s),r) = T_0(\zeta_{0}(s),r) - T_0(\zeta_{0}(s),r_0),\qquad\forall s\in[0,1),\ r\in[r_0,1].
\end{align*}

Equation~\eqref{e:accumulation} implies that
$S_0\subset\overline U_1$.
We claim that actually
\begin{align}
\label{e:S0_subset_U1}
S_0\subset U_1. 
\end{align}
We first complete the proof of the proposition assuming this claim.
By~\eqref{e:S0_subset_U1}, the space
\begin{align*}
 S_1:=\big\{ \phi_{T_1(z,1)}(z)\ \big|\ z\in S_0 \big\}\subset \Wu(\gamma_0)\cap\Sigma_1
\end{align*}
is a smooth embedded circle, and therefore $\diam(S_1)\geq c(\Sigma_1,\gamma_0) >0$ according to Lemma~\ref{l:diam_circle}. Equation~\eqref{e:accumulation} implies that $S_1
\subset
\overline{\zeta_{1}([0,1))}$, and therefore
\begin{align*}
 \diam(\zeta_1([0,1)))=\diam(\overline{\zeta_1([0,1))})\geq\diam(S_1)\geq c,
\end{align*}
where
\[
c:=\min\big\{ c(\Sigma_1,\gamma_0)\ \big|\ \gamma_0\subset \Kbr \big\}>0.
\]

It remains to establish~\eqref{e:S0_subset_U1}. Let us assume by contradiction that~\eqref{e:S0_subset_U1} does not hold, so that there exists 
\begin{align}
\label{e:z_not_in_U1}
z\in S_0\setminus U_1. 
\end{align}
By \eqref{e:accumulation} there is a sequence $s_n\in[0,1)$ such that $s_n\to1$ and $z_n:=\zeta_{r_0}(s_n)\to z$. 
By the implicit function theorem, there exist a maximal $r_1\in(r_0,1]$ and a unique monotone increasing diffeomorphism onto its image $\tau_1:[r_0,r_1)\to[0,\infty)$ such that $\tau_1(r_0)=0$ and $\phi_{\tau_1(r)}(z)\in\Sigma_r$ for all $r\in[r_0,r_1)$.  Since $z_n\to z$, we have
\begin{align}
\label{e:lim_T1_tau1}
\lim_{n\to\infty} T_1(z_n,r)= \tau_1(r),\qquad\forall r\in[r_0,r_1).
\end{align}

We see that $\tau_1(r)\to\infty$ as $r\to r_1$. Indeed, assume by contradiction that $\tau_1$ 
is  bounded, so that 
\begin{align*}
\lim_{r\to r_1} \tau_1(r) = t_1<\infty.
\end{align*}
This, together with~\eqref{e:lim_T1_tau1},  implies that $\phi_{t_1}(z)\in\interior(\Sigma_{r_1})$. Therefore $r_1=1$ due to the maximality of $r_1$. However, this implies that $z\in U_1$, which contradicts~\eqref{e:z_not_in_U1}.

Since $\tau_1(r)\to\infty$ as $r\to r_1$, Lemma~\ref{l:unstable_Ws} implies that $z\in\Ws(\gamma_1)$ for some broken binding component $\gamma_1\subset \Kbr$. Therefore $z\in\Wu(\gamma_0)\cap\Ws(\gamma_1)$. By assumption, the intersection $\Wu(\gamma_0)\cap\Ws(\gamma_1)$ is transverse. This, together with~\eqref{l:lambda_lemma}, implies that there exist arbitrarily small open neighborhoods $W\subset\Ws(\gamma_1)$ of $z$ such that $W\cap V\neq\varnothing$. Therefore, there exists a sequence $z_n''\in\Sigma_{r_0}\cap\Ws(\gamma_1)\cap V$ such that $z_n''\to z$. We choose $\epsilon'>0$ small enough so that the local stable manifold $\Ws_{\epsilon'}(\gamma_1)$ does not intersect $\interior(\Sigma_1)$, i.e.
\begin{align}
\label{e:Ws_disjoint_Sigma1}
 \Ws_{\epsilon'}(\gamma_1)\cap\interior(\Sigma_1)=\varnothing.
\end{align}
There exists $\delta>0$ small enough such that \[z''':=\phi_{\tau_1(r_1-\delta)}(z)\in \Ws_{\epsilon'/2}(\gamma_1).\] 
Since \[z_n''':=\phi_{T_1(z_n'',r_1-\delta)}(z_n'')\to z''',\] in particular $z_n'''\in\Ws_{\epsilon'}(\gamma_1)\cap\interior(\Sigma_{r_1-\delta})$ for all $n$ large enough. However, $\Ws_{\epsilon'}(\gamma_1)$ is positively invariant by the Reeb flow, and therefore $\phi_t(z_n''')\in\Ws_{\epsilon'}(\gamma_1)$ for all $t\geq0$. This, together with~\eqref{e:Ws_disjoint_Sigma1}, contradicts the fact that, for \[t:=T_1(z_n'',1)-T_1(z_n'',r_1-\delta)>0,\] we have 
\[
\phi_{t}(z_n''')=\phi_{T_1(z_n'',1)}(z_n'')\in\interior(\Sigma_{1}). 
\qedhere
\]
\end{proof}

\section{Hyperbolic invariant sets}
\label{s:hyperbolic}

\subsection{Hyperbolicity}
\label{ss:hyperbolicity}

We begin by recalling some elements from hyperbolic dynamics, which generalize those already introduced in Section~\ref{ss:non_degenerate} for periodic orbits. 
We refer the reader to, e.g., the recent monograph of Fisher-Hasselblatt \cite{Fisher:2019vz} for more details.

Let $N$ be a closed manifold equipped with an auxiliary Riemannian metric, $X$ a nowhere vanishing vector field on $N$, $\phi_t:N\to N$ its flow, and $\Lambda\subset N$ a compact subset invariant by the flow $\phi_t$, i.e.~$\phi_t(\Lambda)=\Lambda$ for all $t\in\R$.
The invariant subset $\Lambda$ is \emph{hyperbolic} if each fiber of the restricted tangent bundle $TN|_{\Lambda}$ admits a splitting
\begin{align*}
 T_zN = \Es(z)\oplus\Eu(z)\oplus\mathrm{span}(X(z)),\qquad\forall z\in\Lambda,
\end{align*}
and there exist constants $b,c>0$ with the following three properties:
\begin{itemize}

\item (\emph{Invariance}) $d\phi_t(z)\Es(z)=\Es(\phi_t(z))$ and $d\phi_t(z)\Eu(z)=\Eu(\phi_t(z))$ for all $z\in\Lambda$ and $t\in\R$,

\item (\emph{Forward contraction}) $\|d\phi_t(z)v\|\leq b\,e^{-ct}\|v\|$ for all $z\in\Lambda$, $v\in\Es(z)$, and $t>0$,

\item (\emph{Backward contraction}) $\|d\phi_{-t}(z)v\|\leq b\,e^{-ct}\|v\|$ for all $z\in\Lambda$, $v\in\Eu(z)$, and $t>0$.

\end{itemize}
The norm appearing in the last two points is induced by the auxiliary Riemannian metric on $N$, and thanks to the compactness of $\Lambda$ the notion of hyperbolicity is independent of the choice of such a Riemannian metric. 
It follows from the three properties listed above that $\Es$ and $\Eu$ are continuous sub-bundles of $TN|_\Lambda$, called the \emph{stable} and \emph{unstable} sub-bundles respectively. 
We stress that the positive constants $b$ and $c$ do not depend on the specific point $z\in\Lambda$ (for this reason, some authors prefer the expression ``uniform hyperbolicity'').

We denote by $d:N\times N\to[0,\infty)$ the distance induced by the auxiliary Riemannian metric on $N$.
For each point $z\in\Lambda$, we consider the spaces
\begin{align*}
\Ws(z)&=\Big\{ z'\in N\ \Big|\ \lim_{t\to\infty}d(\phi_t(z'),\phi_t(z))=0 \Big\},\\
\Wu(z)&=\Big\{ z'\in N\ \Big|\ \lim_{t\to-\infty}d(\phi_t(z'),\phi_t(z))=0 \Big\},
\end{align*}
which are injectively immersed submanifolds of $N$ with tangent spaces $T_z\Ws(z)=\Es(z)$ and $T_z\Wu(z)=\Eu(z)$.
The \emph{stable} and \emph{unstable laminations} of $\Lambda$ are defined respectively as
\begin{align*}
\Ws(\Lambda)=\bigcup_{z\in\Lambda} \Ws(z),\qquad
\Wu(\Lambda)=\bigcup_{z\in\Lambda} \Wu(z).
\end{align*}
These spaces are invariant by the flow $\phi_t$.

\subsection{Local product structure}
\label{ss:local_product_structure}
Let $\Lambda\subset N$ be a hyperbolic invariant set of the flow $\phi_t$ of the nowhere vanishing vector field $X$.
For each $\epsilon>0$ small enough and $z\in\Lambda$, consider the subspaces
\begin{align*}
\Ws_{\epsilon}(z) &= \left\{ z'\in\Ws(z)\ \left|\ \displaystyle\sup_{t\geq0} d(\phi_t(z),\phi_t(z'))\leq\epsilon \right.\right\},\\
\Wu_{\epsilon}(z) &= \left\{ z'\in\Wu(z)\ \left|\ \displaystyle\sup_{t\leq0} d(\phi_t(z),\phi_t(z'))\leq\epsilon \right.\right\},
\end{align*}
which are compact embedded submanifolds of $N$ with (possibly non-smooth) boundary. If $\epsilon>0$ is small enough, there exists $\delta>0$ such that, for each pair of points $z_1,z_2\in\Lambda$ with $d(z_1,z_2)<\delta$, there exists a unique $t\in(-\epsilon,\epsilon)$ such that $\Ws_{\epsilon}(\phi_t(z_1))\cap\Wu_{\epsilon}(z_2)\neq\varnothing$. Moreover, this latter intersection is a single point $z_3$, which is denoted by
\begin{align*}
z_3=:\langle z_1,z_2\rangle\in\Ws_{\epsilon}(\phi_t(z_1))\cap\Wu_{\epsilon}(z_2).
\end{align*} 
The map $\langle\,\cdot\,,\cdot\,\rangle$ is continuous on the $\delta$-neighborhood of the diagonal, see \cite[Statement~1.4]{Bowen:1972ws}. The invariant set $\Lambda$ is said to have \emph{local product structure} when the bracket $\langle\,\cdot\,,\cdot\,\rangle$ takes values into $\Lambda$, i.e.~$\langle z_1,z_2\rangle\in\Lambda$ for all $z_1,z_2\in\Lambda$ with $d(z_1,z_2)<\delta$.
It turns out that $\Lambda$ has local product structure if and only if it is a \emph{locally maximal} invariant set, meaning that it has an open neighborhood $U\subset N$, called an \emph{isolating neighborhood}, such that
\begin{align}
\label{e:isolating_nbhd}
 \Lambda = \bigcap_{t\in\R} \phi_t(U),
\end{align}
see \cite[Th.~6.2.7]{Fisher:2019vz} or \cite[App.~E]{Contreras:2014vo}. Under the local maximality assumption, the stable and unstable laminations can be equivalently characterized as stable and unstable basins of $\Lambda$ respectively, i.e.
\begin{equation}
\label{e:basin_characterization}
\begin{split}
\Ws(\Lambda)&=\big\{ z\in N\ \big|\  \omega\mbox{-}\!\lim(z)\subseteq\Lambda\big\},\\
\Wu(\Lambda)&=\big\{ z\in N\ \big|\  \alpha\mbox{-}\!\lim(z)\subseteq\Lambda\big\}, 
\end{split}
\end{equation}
where the $\alpha$-limit and the $\omega$-limit are with respect to the flow $\phi_t$, see \cite[Def.~1.5.5 and Th.~5.3.25]{Fisher:2019vz}.

\subsection{A criterium for the Anosov property}
A \emph{basic set} for a flow is a compact, locally maximal, hyperbolic invariant subset containing a dense orbit and a dense subspace of closed orbits.

\begin{Remark}\label{r:basic_set_with_infinitely_many_closed_orbits}
Let $\Lambda$ be a basic set for the flow of the nowhere vanishing vector field $X$.
The definition of basic set implies that, if $\Lambda$ does not consist of a single closed orbit, then $\Lambda$ must contain infinitely many closed orbits, and every closed orbit in $\Lambda$ is non isolated in $\Per(X)\cap\Lambda$. 
\hfill\qed
\end{Remark}

In the proof of Theorem~\ref{t:Reeb}, we shall need the following result due to Bowen and Ruelle. In the original source \cite[Cor.~5.7]{Bowen:1975ua}, the result is part of a statement for Axiom A flows; nevertheless, the claim that we need holds as well without the Axiom A assumption, and we provide the proof for the reader's convenience. As usual, all our vector fields are smooth, but we point out that Bowen and Ruelle's result requires $C^2$ flows. We recall that a subset $\Lambda$ of a smooth manifold $N$ has positive measure when, for some local chart $\psi:U\to\R^n$ of $N$, the image $\psi(\Lambda\cap U)$ has positive Lebesgue measure.

\begin{Thm}[Bowen-Ruelle]
\label{t:Bowen_Ruelle}
Let $N$ be a closed connected manifold, $X$ a nowhere vanishing $C^2$ vector field on $N$, and $\phi_t:N\to N$ the flow of $X$. If $\Lambda$ is a hyperbolic basic set for $\phi_t$ of positive measure, then $\Lambda=N$.
\end{Thm}

\begin{proof}
Since $\Lambda$ is of positive measure, both $\Ws(\Lambda)$ and $\Wu(\Lambda)$ are of positive measure as well. Since $\Ws(\Lambda)$ is of positive measure, a statement of Bowen and Ruelle \cite[Th.~5.6]{Bowen:1975ua} (see also \cite[Th.~7.4.6]{Fisher:2019vz}) implies that $\Lambda$ is an attractor, i.e.~$\Lambda$ has an open isolating neighborhood $U\subset N$ such that $\phi_t(U)\subset U$ for all $t\geq0$ large enough. Let us consider the characterization~\eqref{e:basin_characterization} of the stable and unstable laminations $\Ws(\Lambda)$ and $\Wu(\Lambda)$.
Since $\Lambda$ is an attractor, an argument due to Hurley \cite[Lemma~1.6]{Hurley:1982uc} implies that there exists an open neighborhood $V\subset N$ of $\Lambda$ such that  
\begin{align*}
 \Ws(\Lambda)=\bigcup_{t>0} \phi_{-t}(V).
\end{align*}
In particular, $\Ws(\Lambda)$ is an open subset of $N$.
Since $\Wu(\Lambda)$ is of positive measure, another application of Bowen and Ruelle's \cite[Th.~5.6]{Bowen:1975ua} implies that $\Lambda$ is an attractor for the backwards flow $\overline{\phi}_t:=\phi_{-t}$, i.e.~$\Lambda$ has an open isolating neighborhood $W\subset N$ such that $\phi_{-t}(W)\subset W$ for all $t\geq0$ large enough. This implies
\begin{align}
\label{e:loc_max_one_sided}
 \bigcap_{t\geq t_0} \phi_{-t}(W)=\Lambda,\qquad\forall t_0\in\R.
\end{align}
For each $z\in\Ws(\Lambda)$ there exists $t_z>0$ such that  $\phi_t(z)\in W$ for all $t\geq t_z$, and by~\eqref{e:loc_max_one_sided} we infer
\begin{align*}
 z\in \bigcap_{t\geq t_z} \phi_{-t}(W)=\Lambda.
\end{align*}
Namely, $\Ws(\Lambda)=\Lambda$, and in particular $\Lambda$ is an open subset of $N$. Since $N$ is connected and $\Lambda$ is both open and closed, we conclude that $\Lambda=N$.
\end{proof}

\subsection{The closure of the space of periodic orbits}
From now on, we will require our closed manifold $N$ to have dimension
$\dim(N)=3$.
We consider the subspace of closed orbits of $\phi_t$, which we denote by
\begin{align*}
\Per(X):=\bigcup_{t>0} \Fix(\phi_t) \subset N.
\end{align*}
In the rest of the section, we will make the following  extra assumptions:
\begin{itemize}

\item[(i)] The closure $\overline{\Per(X)}$ is a hyperbolic invariant subset, whose stable and unstable bundles $\Es,\Eu\subset TN|_{\Lambda}$ have rank 
\begin{align}
\label{e:rank}
\rank(\Es)=\rank(\Eu)=1.
\end{align}

\item[(ii)] The Kupka-Smale transversality condition: $\Wu(\gamma_1)\pitchfork\Ws(\gamma_2)$ for all closed orbits $\gamma_1,\gamma_2\subset\Per(X)$.
\end{itemize}

\begin{Remark}
\label{r:Reeb}
The rank condition~\eqref{e:rank} is satisfied by all hyperbolic invariant subsets of Reeb flows $\phi_t$ on closed contact 3-manifolds $(N,\lambda)$. Indeed, let $\Lambda\subset N$ be such an hyperbolic invariant subset for $\phi_t$, with stable and unstable bundles $\Es,\Eu\subset TN|_\Lambda$. The Reeb flow $\phi_t$ preserves the contact form $\lambda$, and therefore its exterior differential $d\lambda$ as well. This, together with the forward and backward contraction properties of $\Es$ and $\Eu$,  implies that $\Es\oplus\Eu=\ker(\lambda)$, $d\lambda|_{\Es}=0$, and $d\lambda|_{\Eu}=0$. Since $d\lambda$ is symplectic on the distribution $\ker(\lambda)$, we conclude that $\rank(\Es)=\rank(\Eu)=1$.
\hfill\qed
\end{Remark}

The Kupka-Smale transversality condition  implies that, for each pair of closed orbits $\gamma_1,\gamma_2\subset\Per(X)$, 
the intersection $\Wu(\gamma_1)\cap\Ws(\gamma_2)$ is an
 injectively immersed 1-dimensional submanifold of $N$ invariant by the flow; the orbits contained in $\Wu(\gamma_1)\cap\Ws(\gamma_2)$ are called \emph{heteroclinics}, and more specifically \emph{homoclinics} when $\gamma_1=\gamma_2$. 

Two closed orbits $\gamma_1,\gamma_2\subset\Per(X)$ are said to be in the same \emph{homoclinic class}, and we will write it as 
$\gamma_1\sim\gamma_2$,
when the transverse intersections $\Ws(\gamma_1)\cap\Wu(\gamma_2)$ and $\Wu(\gamma_1)\cap\Ws(\gamma_2)$ are both non-empty. 
An application of the shadowing lemma 
from hyperbolic dynamics implies that such heteroclinics are contained in $\overline{\Per(X)}$, i.e.
\begin{align}
\label{e:both_heteroclinics}
 \Big(\Ws(\gamma_1)\cap\Wu(\gamma_2)\Big)
 \cup
 \Big(\Wu(\gamma_1)\cap\Ws(\gamma_2)\Big)
 \subset
 \overline{\Per(X)},
 \qquad
 \forall\gamma_1\sim\gamma_2,
\end{align}
see, e.g., \cite[Cor.~6.5.3]{Fisher:2019vz}. The relation $\sim$ is an equivalence relation, the transitivity property being a consequence of the $\lambda$-lemma from hyperbolic dynamics.

If two points $z_1,z_2\in\Per(X)$ are sufficiently close so that the brackets $\langle z_1,z_2\rangle$ and $\langle z_2,z_1\rangle$ are both well defined, their closed orbits $\gamma_1(t):=\phi_t(z_1)$ and $\gamma_2(t):=\phi_t(z_2)$ are in the same homoclinic class. Therefore, $\langle z_1,z_2 \rangle$ and $\langle z_2,z_1\rangle$ belong to $\overline{\Per(X)}$. By the continuity of the bracket $\langle\,\cdot,\cdot\,\rangle$, we conclude that $\langle z_1,z_2\rangle\in\overline{\Per(X)}$ for all sufficiently close points $z_1,z_2\in\overline{\Per(X)}$. Namely, $\overline{\Per(X)}$ has local product structure, and therefore is locally maximal.

The spectral decomposition theorem of Smale \cite[Th.~5.3.37]{Fisher:2019vz} implies that $\overline{\Per(X)}$ decomposes as a finite disjoint union
\begin{align}
\label{e:spectral_decomposition_Per}
\overline{\Per(X)}
=
\Lambda_1 \cup ...\cup \Lambda_k,
\end{align}
where each $\Lambda_j$ is a basic set. Notice that each $\Lambda_j$ is open in $\Lambda$, since it is the complement of the closed set $\cup_{i\neq j}\Lambda_i$.
The fact that each $\Lambda_j$
contains a dense orbit implies that the decomposition~\eqref{e:spectral_decomposition_Per} is unique.
We will need the following statement statement, which is certainly well known to the experts, and we add its proof for the reader's convenience.

\begin{Lemma}
\label{l:basic_sets}
The basic sets $\Lambda_j$ provided by the spectral decomposition~\eqref{e:spectral_decomposition_Per} are precisely the closures of the homoclinic classes, i.e.
\begin{align*}
\Lambda_j=\overline{\bigcup_{\gamma\sim\gamma_j} \gamma},
\qquad
\forall\mbox{ orbits }
\gamma_j\subset\Lambda_j\cap\Per(X).
\end{align*}
\end{Lemma}

\begin{proof}
We already remarked that two sufficiently close points $z_1,z_2\in\Per(X)$ belong to closed orbits in the same homoclinic class. This, together with the compactness of $\overline{\Per(X)}$, implies that this latter space decomposes as the finite disjoint union
\begin{align}
\label{e:decomposition_homoclinic}
 \overline{\Per(X)}=\Upsilon_1\cup...\cup\Upsilon_h,
\end{align}
where each $\Upsilon_j$ is the closure of a homoclinic class and is open in $\overline{\Per(X)}$. Notice that each $\Upsilon_j$ is a locally maximal invariant set, since it has local product structure. In order to prove the lemma, we are left to show that each $\Upsilon_j$ contains a dense orbit; indeed, this will imply that $\Upsilon_j$ is a basic set, and as we already remarked the decomposition of $\overline{\Per(X)}$ as a finite disjoint union of basic sets is unique. 

Let $\Upsilon\in\{\Upsilon_1,...,\Upsilon_h\}$ be one of the above closure of homoclinic classes, and $U_1,U_2\subseteq\Upsilon$ two non-empty open subsets. We fix points $z_i\in U_i\cap\Per(X)$, and denote by $\gamma_i\subset\Upsilon\cap\Per(X)$ the closed orbit of $z_i$. Let $\gamma\subset \Wu(\gamma_1)\cap\Ws(\gamma_2)$ be a heteroclinic orbit. By~\eqref{e:both_heteroclinics}, $\gamma$ is contained in $\overline{\Per(X)}$. Since $z_1$ belongs to the $\alpha$-limit of $\gamma$ and $z_2$ belongs to the $\omega$-limit of $\gamma$, the heteroclinic $\gamma$ is contained in $\Upsilon$. In particular, there exists $t_1<t_2$ such that $\gamma(t_1)\in U_1$ and $\gamma(t_2)\in U_2$. Therefore, for $t:=t_2-t_1>0$, we have
\begin{align}
\label{e:transitivity}
\phi_t(U_1)\cap U_2\neq\varnothing.
\end{align}

It is well known that the validity of~\eqref{e:transitivity} for arbitrary non-empty open subsets $U_1,U_2\subseteq\Upsilon$ implies that $\Upsilon$ contains a dense orbit. This can be  shown as follows. Let $V_n\subset\Upsilon$, for $n\in\N$, be a countable basis of non-empty open subsets of $\Upsilon$. We choose a non-empty open subset $W_0\subset\Upsilon$ such that $\overline{W_0}\subset V_0$. We now define iteratively for increasing values of $n\geq1$ the open subset
\begin{align*}
 Z_n:= \bigcup_{t>0} \Big( W_{n-1}\cap\phi_t^{-1}(V_n)\Big),
\end{align*}
which is non-empty according to the argument of the previous paragraph. We choose a non-empty open subset $W_n\subset \Upsilon$ such that $\overline{W_n}\subset Z_n$. We have obtained a nested sequence of non-empty sets
\begin{align*}
 W_0
 \supset 
 \overline{W_1}
 \supset 
 W_1
 \supset 
 \overline{W_2}
 \supset 
 W_2
 \supset
 ...
\end{align*}
Cantor's intersection theorem implies that the intersection
\begin{align*}
W:=\bigcap_{n\in\N} W_n = \bigcap_{n\in\N} \overline{W_n}
\end{align*}
is non-empty. For each $z\in W$, the corresponding orbit $\gamma(t):=\phi_t(z)$ visits all the open subsets $V_n$ of the basis, and therefore is a dense orbit in $\Upsilon$.
\end{proof}

Let $\Lambda\subset\overline{\Per(X)}$ be the closure of a homoclinic class.
Every point $z\in\Lambda$ has an arbitrarily small open neighborhood $U\subset N$ such that the path-connected component $W\subset\Ws(\Lambda)\cap U$ containing $z$ separates $U$ in two path-connected components $U_1$ and $U_2$. The point $z$ is called an \emph{$s$-boundary} when at least one $U_i$ satisfies $U_i\cap\Lambda=\varnothing$. 
By replacing the stable lamination $\Ws(\Lambda)$ with the unstable lamination $\Wu(\Lambda)$, we obtain the analogous concept of \emph{$u$-boundary}. 
The following lemma was proved for suitable discrete dynamical systems by Bonatti and Langevin \cite[Prop.~2.1.1]{Bonatti:1998uo}, and we provide a slightly different argument here in the setting of our flows.

\begin{Lemma}
\label{l:bonatti_langevin}
Any $s$-boundary point $z\in\Lambda$ is contained in the stable manifold $\Ws(\gamma)$ of some closed orbit $\gamma\subset\Per(X)\cap\Lambda$. Analogously, any $u$-boundary $z\in\Lambda$ is contained in the unstable manifold $\Wu(\gamma)$ of some closed orbit $\gamma\subset\Per(X)\cap\Lambda$.
\end{Lemma}

\begin{proof}
We prove the first claim, the second one being analogous.  We fix $\epsilon>0$ small enough so that, for each $z\in\Lambda$, the local unstable manifold $\Wu_\epsilon(z)$ is a compact interval embedded in $N$. Let $z\in\Lambda$ be an $s$-boundary point. By the local product structure of $\Lambda$ and the definition of $s$-boundary point, up to reducing $\epsilon>0$ there exists a connected component $\ell\subset\Wu_\epsilon(z)\setminus\{z\}$ such that $z\in\partial\ell$ and $\ell\cap\Lambda=\varnothing$. 
Notice that, for all $t\in\R$, $\phi_t(\ell)$ is an interval in $\Wu(\phi_t(z))\setminus\{\phi_t(z)\}$ such that $\phi_t(z)\in\partial(\phi_t(\ell))$.

If needed, we further reduce the constant $\epsilon>0$ so that we can employ it in the definition of local product structure of $\Lambda$: there exists $\delta>0$ such that, for each pair of points $z_1,z_2\in\Lambda$ with $d(z_1,z_2)<\delta$, there exists a unique $t\in(-\epsilon,\epsilon)$ and a well defined bracket $\langle z_1,z_2\rangle\in\Ws_{\epsilon}(\phi_t(z_1))\cap\Wu_{\epsilon}(z_2)$.

Let $z'\in\Lambda$ be an $\omega$-limit of the orbit $t\mapsto\phi_{t}(z)$. Namely, there exists a sequence $t_n\to\infty$ such that 
\begin{align}
\label{e:convergence_BL}
z_n:=\phi_{t_n}(z)\to z'.
\end{align}
Up to removing finitely many elements from the sequence $z_n$, we can assume that $d(z_n,z')<\delta/2$, so that the brackets $\langle z_n,z'\rangle$ and $\langle z_n,z_m\rangle$ are well defined for all $n,m\in\N$. 
We claim that, up to further removing finitely many elements from the subsequence $z_n$, we have
\begin{align}
\label{e:same_projection}
\langle z_n,z' \rangle=z',\qquad\forall n\in\N. 
\end{align}
Indeed, the continuity of the bracket $\langle\,\cdot,\cdot\,\rangle$ implies that $z_n':=\langle z_n,z' \rangle\to z'$. If~\eqref{e:same_projection} does not hold for all but finitely many values of $n$, then the sequence $z_n'$ contains infinitely many pairwise distinct points. In particular, we can find $n_1<n_2<n_3$ such that the three points  $z_{n_1}', z_{n_2}', z_{n_3}'$ are pairwise distinct, and $z_{n_2}'$ lies in the interior of a compact interval $\ell'\subset\Wu_\epsilon(z')$ with boundary $\partial\ell'=\{z_{n_1}', z_{n_3}'\}$. This implies that the points $\langle z_{n_1},z_{n_2}\rangle,\langle z_{n_3},z_{n_2}\rangle\in\Lambda\cap\Wu_\epsilon(z_{n_2})$ belong to different path-connected components $\ell_1,\ell_2\subset\Wu_\epsilon(z_{n_2})\setminus\{z_{n_2}\}$, see Figure~\ref{f:projections}. Notice that $\phi_{-t_{n_2}}(\Wu_\epsilon(z_{n_2}))\subset\Wu_\epsilon(z)$, and the intervals $\phi_{-t_{n_2}}(\ell_1)$ and $\phi_{-t_{n_2}}(\ell_2)$ belong to different path-connected components of $\Wu_\epsilon(z)\setminus\{z\}$. Therefore for one value of $i\in\{1,2\}$ we must have $\phi_{-t_{n_2}}(\ell_i)\subset\ell$. This gives a contradiction, since $\phi_{-t_{n_2}}(\ell_i)\cap\Lambda\neq\varnothing$ but $\ell\cap\Lambda=\varnothing$.

\begin{figure}
\begin{footnotesize}
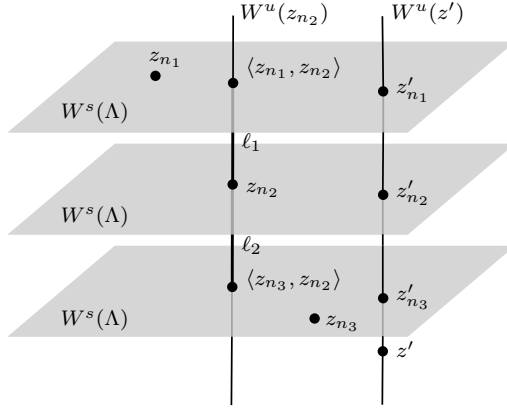
\end{footnotesize}
\caption{Projections $z_{n_i}'$ to the unstable manifold $\Wu(z')$.}
\label{f:projections}
\end{figure}

Having established~\eqref{e:same_projection}, we can now assume without loss of generality, up to modifying the values $t_n$ without affecting the conditions $t_n\to\infty$ and $z_n\to z'$, that
\begin{align}
\label{e:Wszn_cap_Wuz}
 \Ws_\epsilon(z_n)\cap\Wu_\epsilon(z') =\{z'\}.
\end{align} 
In particular
\begin{align}
\label{e:same_stable_mfds}
\Ws(z_n)=\Ws(z'). 
\end{align}
Since $z_{n+1}=\phi_{t_{n+1}-t_n}(z_n)$, we have that
\begin{align*}
 \phi_{t_{n+1}-t_n}(\Ws(z'))=
 \phi_{t_{n+1}-t_n}(\Ws(z_n))=
 \Ws(z_{n+1})=
 \Ws(z').
\end{align*}
For each $n_0\in\N$, the convergence 
\begin{align*}
 \lim_{n\to\infty} d(z_n,\phi_{t_n-t_{n_0}}(z'))= \lim_{n\to\infty} d(\phi_{t_n-t_{n_0}}(z_{n_0}),\phi_{t_n-t_{n_0}}(z'))=0,
\end{align*}
together with~\eqref{e:convergence_BL}, implies
\begin{align*}
 \lim_{n\to\infty} \phi_{t_n-t_{n_0}}(z') = z',\qquad\forall n_0\in\N.
\end{align*}

Equation~\eqref{e:Wszn_cap_Wuz} implies that $z_n\in\Ws_\epsilon(z')$ and then also $\phi_{t_n-t_{n_0}}(z')\in W^s_\epsilon(z')$. 
Up to extracting a subsequence, we can assume that all the $z_n$ belong to the same path-connected component $W_1$ of $\Ws_\epsilon(z')\setminus\{z'\}$. Let $W_2:=\Ws_\epsilon(z')\setminus(\{z'\}\cup W_1)$ be the other path-connected component, and $n_0$ an integer. 
Up to extracting a further subsequence, we can assume that there exists $j\in\{1,2\}$ such that
\begin{align*}
\phi_{t_{n}-t_{n_0}}(z')\in W_j\cup\{z'\},\qquad\forall n>n_0
\end{align*}
We consider the two possible cases:
\begin{itemize}
\item If $j=1$, we denote by $\ell''\subset\Ws(z')$ the compact interval with boundary $\partial\ell''=\{z_{n_0},z'\}$. We choose a large enough integer $n_1>n_0$ so that both points $z_{n_1}$ and $\phi_{t_{n_1}-t_{n_0}}(z')$ belong to $\ell''$, and therefore  $\phi_{t_{n_1}-t_{n_0}}(\ell'')\subset\ell''$.  Brouwer fixed point theorem implies that there exists a fixed point 
\[z''\in\ell''\cap\Fix(\phi_{t_{n_1}-t_{n_0}}).\]

\item If $j=2$, we fix any integer $n_1>n_0$. If $\phi_{t_{n_1}-t_{n_0}}(z')=z'$, we set
\begin{align*}
 z'':=z'.
\end{align*}
If $\phi_{t_{n_1}-t_{n_0}}(z')\neq z'$,  we have $\phi_{t_{n_1}-t_{n_0}}(z')\in W_2$, and we denote by $\ell''\subset\Ws(z')$ the compact interval with boundary $\partial\ell''=\{z_{n_1},\phi_{t_{n_1}-t_{n_0}}(z')\}$. We then choose a large enough integer $n_2>n_1$ so that  both points $z_{n_2}$ and $\phi_{t_{n_2}-t_{n_0}}(z')$ belong to $\ell''$, and therefore  $\phi_{t_{n_2}-t_{n_0}}(\ell'')\subset\ell''$. Brouwer fixed point theorem implies that there exists a fixed point $z''\in\ell''\cap\Fix(\phi_{t_{n_2}-t_{n_0}})$.

\end{itemize}
In both cases, we obtained a point $z''\in\Per(X)\cap\Ws(z')$. Therefore $z'\in\Ws(z'')$, which implies that $z''\in\Lambda$ and $\Ws(z'')=\Ws(z')$. This, together with~\eqref{e:same_stable_mfds}, implies that $z\in\Ws(\gamma)$.
\end{proof}

The next lemma was established by Contreras and Oliveira \cite[Lem\-ma~3.2]{Contreras:2004vh}. Since our setting is slightly different, we provide a detailed proof for the reader's convenience.

\begin{Lemma}
\label{l:Ws_cap_Wu}
Let $\Lambda\subset\overline{\Per(X)}$ be the closure of a homoclinic class. Then
\[ \Ws(\Lambda)\cap\Wu(\Lambda)=\Lambda. \]
\end{Lemma}

\begin{proof}
We consider an arbitrary point $z_0\in \Ws(\Lambda)\cap\Wu(\Lambda)$. We need to prove that $z_0\in\Lambda$. We already remarked in Equation~\eqref{e:both_heteroclinics} that this holds if $z_0\in \Ws(\gamma_1)\cap\Wu(\gamma_2)$ for some closed orbits $\gamma_1,\gamma_2\subset\Per(X)\cap \Lambda$. 
Therefore, it remains to consider the case in which $z_0$ does not belong to the intersection of the stable and unstable manifolds of two closed orbits in $\Lambda$. Let us consider the case in which $z_0$ does not belong to the stable manifold of a closed orbit in $\Lambda$, the other case being analogous. Let $z'\in\Lambda\setminus\Per(X)$ and $z''\in\Lambda$ be two points such that $z_0\in \Ws(z')\cap\Wu(z'')$. Without loss of generality, we can assume that $z_0$ belongs to the interior of the local stable manifold $\Ws_\epsilon(z')$ for an arbitrarily small $\epsilon>0$; for this, it suffices to replace $z_0$, $z'$, and $z''$ with $\phi_t(z_0)$, $\phi_t(z')$, and $\phi_t(z'')$ respectively for a large time $t>0$. 

A priori, the path-connected component $\ell_0\subset\Ws_\epsilon(z')\cap\Wu(z'')$ containing $z_0$ might be a compact interval that is not just the singleton $\{z_0\}$. The point $z'$ cannot be contained in $\ell_0$, since the stable and unstable laminations $\Ws(\Lambda)$ and $\Wu(\Lambda)$ intersect transversely at $z'\in\Lambda$. We consider the point $z_1\in\partial\ell_0\cap\Ws_\epsilon(z')$ that is closest to $z'$. 
In order to complete the proof, it is enough to show that
\begin{align*}
 z_1\in\Lambda,
\end{align*}
since from this we would immediately infer that  $T_{z_1}\Ws(z')\neq T_{z_1}\Wu(z'')$, and therefore that $\ell_0=\{z_1\}=\{z_0\}$ and $z_0\in\Lambda$. 

\begin{figure}
\begin{footnotesize}
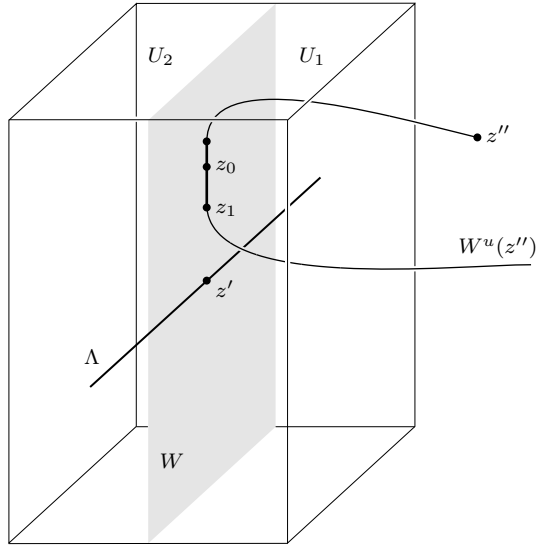
\end{footnotesize}
\caption{The open set $U=U_1\cup W\cup U_2$ with an interval of heteroclinic intersections with boundary on $z_1$.}
\label{f:intersection}
\end{figure}

We set
\begin{align*}
 W:=\bigcup_{t\in(-\epsilon,\epsilon)} \interior(\Ws_\epsilon(\phi_t(z'))),
\end{align*}
and we consider a tubular neighborhood $U\subset N$ of $W$, so that $W$ separates $U$ into two path-connected components $U_1$ and $U_2$. Let $\ell_1\subset\Wu(z'')\cap U$ be the path-connected component containing $z_1$, and notice that $\ell_0\subset\ell_1$. If $z_1$ were a point of transverse intersection of $\Ws(\Lambda)$ and $\Wu(\Lambda)$, $\ell_1$ would intersect both $U_1$ and $U_2$. In general, since we chose $z_1$ in the boundary $\partial\ell_0$,  $\ell_1$ intersects at least one of these two subsets, say $U_1$, in any neighborhood of $z_1$ (Figure~\ref{f:intersection}).
Since $z_0$ does not belong to the stable manifold of a closed orbit in $\Lambda$, the same is true for $z_1$ and $z'$. Lemma~\ref{l:bonatti_langevin} implies that $z'$ is not an $s$-boundary, and therefore belongs to both $\overline{\Per(X)\cap\Lambda\cap U_1}$ and $\overline{\Per(X)\cap\Lambda\cap U_2}$. This allows us to choose a sequence $z_n'\in\Per(X)\cap\Lambda\cap U_1$ such that $z_n'\to z'$. We denote by $W_n\subset U\cap \Ws(\Lambda)$ the path-connected component  containing $z_n'$. Since $\Ws(\Lambda)$ is a lamination, $W_n$ does not intersect $W$, and therefore is contained in $U_1$. Moreover, $W_n$ tends to $W$ as $n\to\infty$, in the sense that
\begin{align*}
 W=\bigcap_{m\in\N} \overline{\bigcup_{n\geq m} W_n}.
\end{align*}
We choose a sequence $z''_n\in\Per(X)\cap\Lambda$ such that $z_n''\to z''$. The unstable manifold $\Wu(z_n'')$ accumulates on $\Wu(z'')$, meaning
\begin{align*}
 \Wu(z'')\subseteq\bigcap_{m\in\N} \overline{\bigcup_{n\geq m} \Wu(z_n'')}.
\end{align*}
For any neighborhood $V\subset U$ of $z_1$ and for all $n$ large enough, the surface
 $W_n$ intersects $\ell_1\cap V$, and therefore intersects $\Wu(z_m'')\cap V$ as well for all $m$ large enough. This implies that, up to extracting subsequences of $z_n'$ and $z_n''$, we have a sequence of  intersection points $z_n\in W_n\cap\Wu(z_n'')$ such that $z_n\to z_1$. Since $z_n',z_n''\in\Per(X)$, we have  $W_n\cap\Wu(z_n'')\subset\Lambda$. Since $\Lambda$ is compact and $z_n\in\Lambda$, we infer that $z_1\in\Lambda$ as well.
\end{proof}

\subsection{Heteroclinic rectangles}
\label{ss:heteroclinic_rectangle}

The stable and unstable laminations of a hyperbolic invariant set might not be closed subspaces of $N$. In this subsection, we essentially show that they are locally closed at least near the closure of a homoclinic class $\Lambda$ containing infinitely many closed orbits. We do not know if such a statement holds in the general setting of this section. Our proof will be valid in the contact setting, and will employ a broken book decomposition as in Section~\ref{ss:broken_book_decomposition}.

Let $\Sigma\subset N$ be a surface transverse to the vector field $X$. A \emph{heteroclinic rectangle} of $\Lambda$ is a compact 2-disk $D\subset\Sigma$ whose boundary is the union of four pairwise distinct smooth compact intervals 
\[\partial D=\ell_1\cup\ell_2\cup\ell_3\cup\ell_4,\] 
with $\ell_1,\ell_3\subset\Ws(\Lambda)$, $\ell_2,\ell_4\subset\Wu(\Lambda)$, and $\ell_i\cap\ell_{i+1}=\{z_{i+1}\}$ for all $i\in\Z/4\Z$, see Figure~\ref{f:heteroclinic}. By Lemma~\ref{l:Ws_cap_Wu}, every corner $z_i$ belongs to the invariant set $\Lambda$, and therefore it is a transverse intersection between $\Ws(\Lambda)$ and $\Wu(\Lambda)$.

\begin{figure}
\begin{footnotesize}
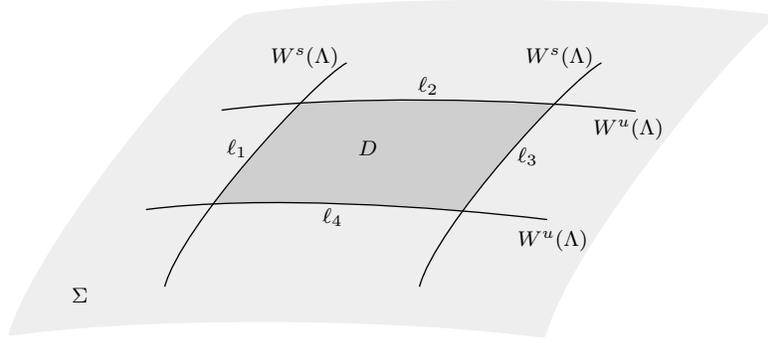
\end{footnotesize}
\caption{A heteroclinic rectangle.}
\label{f:heteroclinic}
\end{figure}

\begin{Lemma}
\label{l:Ws_Wu_closed}
Let $(N,\lambda)$ be a closed contact 3-manifold whose Reeb vector field $X=X_\lambda$ satisfies the following two conditions:
\begin{itemize}
\item[$(i)$] The closure $\overline{\Per(X)}$ is hyperbolic.\vspace{3pt}

\item[$(ii)$] The Kupka-Smale transversality condition holds: $\Ws(\gamma_1)\pitchfork\Wu(\gamma_2)$ for all closed Reeb orbits $\gamma_1,\gamma_2\subset\Per(X)$.

\end{itemize}
Let $\Lambda\subset\overline{\Per(X)}$ be the closure of a homoclinic class containing infinitely many closed Reeb orbits, $\Sigma$ a page of a broken book decomposition of $(N,\lambda)$,  and $D\subset\interior(\Sigma)$ a heteroclinic rectangle of $\Lambda$ with sufficiently small diameter. Then $\Ws(\Lambda)\cap D$ and $\Wu(\Lambda)\cap D$ are compact, and their union is path-connected.
\end{Lemma}

\begin{proof}
Since $\interior(\Sigma)$ is transverse to the Reeb vector field $X$, the intersections $\interior(\Sigma)\cap\Ws(\Lambda)$ and $\interior(\Sigma)\cap\Wu(\Lambda)$ are transverse, and therefore their path-connected components are 1-dimensional manifolds.
Let $\epsilon>0$ be small enough so that, for each $z\in\Lambda$, the spaces $\Ws_{\epsilon}(z)$ and $\Wu_{\epsilon}(z)$ are embedded compact 1-dimensional submanifolds with boundary (i.e.~compact intervals) of $N$. 
For each invariant subset $V\subset\Lambda$, we employ the  notation
\begin{align*}
 \Ws_{\epsilon}(V):=\bigcup_{z\in V} \Ws_{\epsilon}(z),
\qquad
\Wu_{\epsilon}(V):=\bigcup_{z\in V} \Wu_{\epsilon}(z)
.
\end{align*}
Notice that, since $\Lambda$ is compact, the subsets $\Ws_{\epsilon}(\Lambda)$ and $\Wu_{\epsilon}(\Lambda)$ are compact as well.

We denote the Reeb orbit of a point $z\in N$ by 
\[\gamma_z(t):=\phi_t(z).\]
We require the diameter $\diam(D)$ to be small enough so that, for each $z\in D\cap\Lambda$, the path-connected component $\ell\subset\Ws(\Lambda)\cap D$ containing $z$ is actually contained in $\Ws_{\epsilon}(\gamma_z)$, and is a compact interval with boundary $\partial\ell\subset\partial D$; notice that, in this case, we must have $\partial\ell\subset\interior(\ell_2\cup\ell_4)$, where $\ell_2$ and $\ell_4$ are the two smooth segments in $\partial D\cap\Wu(\Lambda)$. Analogously, we require the path-connected component $\ell'\subset\Wu(\Lambda)\cap D$ containing $z$ to actually be contained in $\Wu_{\epsilon}(\gamma_z)$, and to be a compact interval with boundary $\partial\ell\subset\interior(\ell_1\cup\ell_3)$, where $\ell_1$ and $\ell_3$ are the two smooth segments in $\partial D\cap\Ws(\Lambda)$.

Let $\Sigma_1,...,\Sigma_h$ be a finite collection of pages of a broken book decomposition of $(N,\lambda)$ such that each Reeb orbit intersects at least one of these pages. We further require the diameter $\diam(D)$ to be small enough so that
\begin{align}
\label{e:D_small_diameter}
 \diam(D)<\min\big\{c(\Sigma_j,\Sigma)\ \big|\ j=1,...,h\big\},
\end{align}
where the positive constants $c(\Sigma_j,\Sigma)>0$ are those provided by Proposition~\ref{p:large_return}.

Since $\Ws(\Lambda)\cap\Wu(\Lambda)=\Lambda$ according to Lemma~\ref{l:Ws_cap_Wu}, we have
\begin{align*}
\Wu(\Lambda)\cap \partial D = \Wu_{\epsilon}(\Lambda)\cap \partial D,
\qquad
\Ws(\Lambda)\cap \partial D = \Ws_{\epsilon}(\Lambda)\cap \partial D.
\end{align*}
Consider now an arbitrary point $z'\in \interior(D)\cap\Wu(\Lambda)$. We must have $z'\subset\Wu(z'')$ for some $z''\in \Lambda$. Since $\Lambda$ is a basic set (Lemma~\ref{l:basic_sets}) containing infinitely many closed Reeb orbits, every closed Reeb orbit in $\Lambda$ is non isolated in $\Per(X)\cap\Lambda$ (Remark~\ref{r:basic_set_with_infinitely_many_closed_orbits}).
 Therefore, we can find two sequences $y_n''\in\Per(X)\cap\Lambda\setminus K$ and $y_n'\in\Wu(y_n'')$ such that $y_n'\to z'$ and $y_n''\to z''$, where $K\subset N$ is the binding of the broken book decomposition of $(N,\lambda)$. Since $y_n'\to z'$, there exists a sequence $t_n\to0$ such that $\phi_{t_n}(y_n')\in D$. We set 
\begin{align*}
z_n'':=\phi_{t_n}(y_n'')\in \Per(X)\cap\Lambda\setminus K,
\qquad
z_n':=\phi_{t_n}(y_n')\in D\cap\Wu(z_n''),
\end{align*}
and notice that $z_n'\to z'$ and $z_n''\to z''$.
We fix a value of $n$, and a page $\Sigma_j\in\{\Sigma_1,...,\Sigma_h\}$ whose interior intersects $\gamma_{z_n''}$. Since $d(\phi_{-t}(z_n'),\phi_{-t}(z_n''))\to0$ as $t\to\infty$, there exists an arbitrarily large $t>0$ and a smooth embedded path
\[\zeta_0:[0,1]\hookrightarrow\Sigma_j\cap\Wu_{\epsilon}(\gamma_{z_n''})\]
such that $\zeta_0(0)=\phi_{-t}(z_n')$ and $\zeta_0(1)\in\gamma_{z_n''}$. 
Notice that
\[\dot\zeta_0(1)\pitchfork\Ws(\gamma_{z_n''}).\]
By the implicit function theorem, there exist a maximal $s_0\in(0,1]$ and a smooth function $\tau:[0,s_0)\to(0,\infty)$ such that $\tau(0)=t$ and $\zeta_1(s):=\phi_{\tau(s)}(\zeta_0(s))\in D$ for all $s\in[0,s_0)$.
Notice that the smooth path $\zeta_1:[0,s_0)\to D$ satisfies $\zeta_1(0)=z_n'$, and is contained in the unstable lamination $\Wu(\Lambda)$.
The function $\tau$ must be bounded, otherwise Proposition~\ref{p:large_return} would imply that $\diam(\zeta_1([0,s_0)))>c(\Sigma_j,\Sigma)$, contradicting~\eqref{e:D_small_diameter}. Therefore,  $\tau$ admits an extension to a smooth function of the closed interval $\tau:[0,s_0]\to(0,\infty)$, and $\zeta_1$ admits a smooth extension 
$\zeta_1:[0,s_0]\to \Wu(\Lambda)\cap D$. Since $s_0$ was maximal, we have two possible cases: 
\begin{itemize}

\item $s_0=1$, and therefore $\zeta_1(1)\in\gamma_{z_n''}\subset\Lambda$;\vspace{3pt}

\item $s_0<1$, and therefore $\zeta_1(s_0)\in\Wu(\Lambda)\cap(\ell_1\cup\ell_3)\subset \Ws(\Lambda)\cap\Wu(\Lambda)=\Lambda$.

\end{itemize}
In both cases, $\zeta_1(s_0)\in\Lambda\cap D$, which implies 
$z_n'\in \Wu_{\epsilon}(\Lambda)\cap D$.
Since $\Wu_{\epsilon}(\Lambda)$ is compact and $z_n'\to z'$, we conclude that $z'\in\Wu_{\epsilon}(\Lambda)$. Summing up, we have proved that
\begin{align*}
\Wu(\Lambda)\cap D = \Wu_{\epsilon}(\Lambda)\cap D,
\end{align*}
and therefore $\Wu(\Lambda)\cap D$ is compact. An analogous argument shows that 
\[\Ws(\Lambda)\cap D=\Ws_{\epsilon}(\Lambda)\cap D\] 
is compact, and therefore $(\Ws(\Lambda)\cup\Wu(\Lambda))\cap D$ is compact as well.

Finally, $(\Ws(\Lambda)\cup\Wu(\Lambda))\cap D$ is path-connected. Indeed, it contains the path-connected boundary $\partial D$. For each $z\in D\cap\Wu(\Lambda)$, the path-connected component $\ell\subset D\cap\Wu(\Lambda)$ containing $z$ has boundary $\partial\ell\subset\ell_1\cup\ell_3\subset\partial D$. Analogously,  for each $z'\in D\cap\Ws(\Lambda)$, the path-connected component $\ell'\subset D\cap\Ws(\Lambda)$ containing $z'$ has boundary $\partial\ell'\subset\ell_2\cup\ell_4\subset\partial D$.
\end{proof}

\section{Proof of the main results}
\label{s:proofs}

\subsection{Reeb flows}
\label{ss:proof_Reeb}
We now have all the ingredients needed to carry out the proof of our main general result for Reeb flows.

\begin{proof}[Proof of Theorem~\ref{t:Reeb}]
Let $(N,\lambda)$ be a closed connected contact manifold with Reeb vector field $X=X_\lambda$, Reeb flow $\phi_t:N\to N$, and satisfying the assumptions (i) and (ii) of the theorem: the closure of the subspace of closed Reeb orbits $\overline{\Per(X)}$ is hyperbolic, and the stable and unstable manifolds of closed Reeb orbits intersect transversely. The first assumption implies in particular that all closed Reeb orbits are hyperbolic, and therefore non-degenerate. By a recent result of Colin, Dehornoy, and Rechtman~\cite[Theorem~1.2]{Colin:2020tl}, which generalizes an earlier result of Cristofaro-Gardiner, Hutchings, and Pomerleano \cite{Cristofaro-Gardiner:2019wf}, the contact manifold $(N,\lambda)$ has either two or infinitely many closed Reeb orbits. However, it cannot have only two closed Reeb orbits, since in this case another recent result of Cristofaro-Gardiner, Hryniewicz, Hutchings, and Liu \cite[Th.~1.2]{Cristofaro-Gardiner:2021tc} would imply that the two closed Reeb orbits are elliptic. We conclude that $(N,\lambda)$ has infinitely many closed Reeb orbits\footnote{In the case in which $(N,\lambda)$ is the unit tangent bundle of a closed Riemannian surface $(M,g)$, so that the Reeb flow of $(N,\lambda)$ is the geodesic flow of $(M,g)$, the existence of infinitely many closed Reeb orbits (that is, closed geodesics of $(M,g)$) holds unconditionally, without the need of assumptions (i) and (ii). Indeed, for closed surfaces $M$ of positive genus it is an elementary result: infinitely many closed geodesics can be found as the shortest closed curves in infinitely many homotopy classes of free loops that are not multiples of one another. For $M=S^2$, the existence of infinitely many closed geodesics is a deep result that follows by combining the work of Bangert \cite{Bangert:1993ue} with the work of Franks \cite{Franks:1992wu} or, alternatively, Hingston \cite{Hingston:1993wb}.}.

We consider Smale's spectral decomposition of $\overline{\Per(X)}$ as a finite disjoint union
\begin{align*}
\overline{\Per(X)}
=
\Lambda_1 \cup ...\cup \Lambda_k,
\end{align*}
where each $\Lambda_j$ is a basic set. Since there are infinitely many closed Reeb orbits, at least one of the basic sets $\Lambda\in\{\Lambda_1,...,\Lambda_k\}$ contains infinitely many closed Reeb orbits. We fix such a basic set from now on.

Let us assume by contradiction that the Reeb flow of $(N,\lambda)$ is not Anosov. By Bowen and Ruelle's Theorem~\ref{t:Bowen_Ruelle}, the basic set $\Lambda$ must be of measure zero in the closed manifold $N$. Consider the stable and unstable laminations $\Ws(\Lambda)$ and $\Wu(\Lambda)$. All points in the complement $\Ws(\Lambda)\setminus\Lambda$ have $\omega$-limit in $\Lambda$, and all points in the complement $\Wu(\Lambda)\setminus\Lambda$ have $\alpha$-limit in $\Lambda$. In particular, none of the points in $(\Ws(\Lambda)\cup\Wu(\Lambda))\setminus\Lambda$ is recurrent. By Poincar\'e's recurrence theorem,  $(\Ws(\Lambda)\cup\Wu(\Lambda))\setminus\Lambda$ is of measure zero in $N$, and therefore $\Ws(\Lambda)\cup\Wu(\Lambda)$ is of measure zero as well.

We now consider a broken book decomposition of our closed contact manifold $(N,\lambda)$, as introduced in Section~\ref{ss:broken_book_decomposition}. Since the basic set $\Lambda$ contains infinitely many closed Reeb orbits, there exists a page of the broken book $\Sigma\subset N$ such that 
\[\interior(\Sigma)\cap\Lambda\neq\varnothing.\]
We fix such a page, and a point $z_1\in\Per(X)\cap\interior(\Sigma)\cap\Lambda$ on a closed Reeb orbit $\gamma_1(t)=\phi_t(z_1)$. Notice that such a point $z_1$ exists, for $\Per(X)\cap\Lambda$ is dense in $\Lambda$. Moreover, since the basic set $\Lambda$ contains infinitely many closed Reeb orbits, $\gamma_1$ is non-isolated in $\Per(X)\cap\Lambda$ (Remark~\ref{r:basic_set_with_infinitely_many_closed_orbits}). We then choose a second point $z_2\in\Per(X)\cap\interior(\Sigma)\cap\Lambda$ such that $d(z_1,z_2)<\epsilon$ lying on a closed Reeb orbit $\gamma_2(t)=\phi_t(z_2)$ distinct from $\gamma_1$. Here, $d:N\times N\to[0,\infty)$ denotes an auxiliary Riemannian distance, and $\epsilon>0$ is a small enough positive constant that we will fix in the course of the proof.

We first require $\epsilon>0$ to be small enough so that the brackets $\langle z_1,z_2\rangle$ and $\langle z_2,z_1\rangle$ are well defined and, for some times $t_1,t_2\in\R$, the points $w_1:=\phi_{t_1}(\langle z_1,z_2\rangle)$ and $w_2:=\phi_{t_2}(\langle z_2,z_1\rangle)$ are contained in $\interior(\Sigma)$. Here, $t_1$ and $t_2$ may not be unique, and in such a case we choose them having the smallest possible absolute value, so that $t_1\to0$ and $t_2\to0$ as $\epsilon\to0$. We further require $\epsilon$ to be small enough so that, for each $i,j\in\{1,2\}$ with $i\neq j$, $w_i$ and $z_i$ can be joined by a compact smooth curve $\ell_{ii}\subset \Ws(\gamma_i)\cap\interior(\Sigma)$,  $w_i$ and $z_j$ can be joined by a compact smooth curve $\ell_{ij}\subset \Wu(\gamma_j)\cap\interior(\Sigma)$, and the union $\ell_{11}\cup\ell_{12}\cup\ell_{22}\cup\ell_{21}$ is the piecewise smooth boundary of a compact 2-disk $D\subset\interior(\Sigma)$. With the terminology of Section~\ref{ss:heteroclinic_rectangle}, $D$ is a heteroclinic rectangle of $\Lambda$. Notice that $\diam(D)\to0$ as $\epsilon\to0$. We require $\epsilon>0$ to be small enough so that the following three conditions are verified:
\begin{itemize}

\item[(a)] $\diam(D)<c(\Sigma,\Sigma)$, where $c(\Sigma,\Sigma)>0$ is the constant given by Proposition~\ref{p:large_return};

\item[(b)] the bracket $\langle y_1,y_2\rangle$ is defined for any pair of points $y_1,y_2\in D\cap\overline{\Per(X)}$;

\item[(c)] the space $\big(\Ws(\Lambda)\cup\Wu(\Lambda)\big)\cap D$ is compact and path-connected, according to Lemma~\ref{l:Ws_Wu_closed}.

\end{itemize}
Since $\Lambda$ is the closure of a homoclinic class (Lemma~\ref{l:basic_sets}), condition~(b) implies
\begin{align}
\label{e:D_cap_Per_in_Lambda}
 D\cap\overline{\Per(X)}\subset\Lambda.
\end{align}
We choose an arbitrary connected component $D'$ of $D\setminus (\Ws(\Lambda)\cup\Wu(\Lambda))$, which is an open subset of $\interior(D)$. 
Since the complement $D\setminus D'$ is path-connected, $D'$ is an open 2-disk.

Since $D'$ is a non-empty open subset, by Poincar\'e's recurrence there exists a point $z_0\in D'$ such that $\phi_{t_0}(z_0)\in D'$ for some $t_0>0$. By the implicit function theorem, there exists a maximal connected open neighborhood $U\subset\interior(\Sigma)$ of $z_0$ and a smooth function
$\tau:U\to(0,\infty)$
such that $\tau(z_0)=t_0$ and $\phi_{\tau(z)}(z)\in\Sigma$ for all $z\in U$. As we already remarked at the end of Section~\ref{ss:broken_book_decomposition}, the associated arrival map
\begin{align*}
\psi:U\to\Sigma,\qquad \psi(z)=\phi_{\tau(z)}(z)
\end{align*}
is a diffeomorphism onto its image that preserves the area form $d\lambda$, i.e.\ $\psi^*d\lambda=d\lambda|_U$. Notice that
$\psi(z_0)=\phi_{t_0}(z_0)\in D'$.
Let $U'$ be the connected component of $D'\cap U$ containing $z_0$. We claim that 
\begin{align}
\label{e:psi_n_U_in_D}
\psi(U')\subset D'. 
\end{align}
Indeed, $\psi(U')$ is connected and contains the point $\psi(z_0)$. If  $\psi(U')$ were not contained in $D'$, there would exist $z\in U'$ such that $\psi(z)\in\partial D'\subset \Ws(\Lambda)\cup\Wu(\Lambda)$, which is impossible since $D'\cap(\Ws(\Lambda)\cup\Wu(\Lambda))=\varnothing$ and $\Ws(\Lambda)\cup\Wu(\Lambda)$ is invariant by the Reeb flow.

We claim that
\begin{align}
\label{e:tau_n_unbounded}
 (D'\setminus U')\cup(D'\setminus\psi(U'))\neq\varnothing.
\end{align}
Indeed, if we had the equalities $U'=\psi(U')=D'$, we would have a diffeomorphism $\psi|_{D'}:D'\to D'$ that preserves the area form $d\lambda$. Since the integral of $d\lambda$ on the open 2-disk $D'$ is finite, Brouwer's plane translation theorem would imply the existence of a fixed point $z=\psi(z)\in D'$, which would therefore be contained in $D'\cap\Per(X)$, and therefore in $\Lambda$ according to~\eqref{e:D_cap_Per_in_Lambda}. This would give a contradiction, since $D'\cap\Lambda=\varnothing$.

Assume first that
\begin{align}
\label{e:added_first_subcase_D}
D'\setminus U'\neq\varnothing,
\end{align}
and fix a point $y_1\in\partial U'\cap D'$.
We identify the closure $\overline{D'}$ with a compact disk in $\R^2$, so that we can employ cartesian coordinates. We choose a second point $y_2\in U'$ that is close enough to $y_1$ so that
\begin{align}
\label{e:close_to_z1}
|y_2-y_1|<\min_{y\in\partial D'}|y_2-y|.
\end{align}
We denote by $B_r\subset \overline{D'}$ the open ball of radius $r>0$ centered at $y_2$, i.e.
\begin{align*}
 B_r=\Big\{y\in \overline{D'}\ \Big|\ |y_2-y|< r\Big\}.
\end{align*}
We fix $r>0$ to be the maximal positive radius such that $B_r\subset U'$. By~\eqref{e:close_to_z1}, there exists a point $y_3\in (D'\setminus U')\cap\partial B_r$. We consider the smooth path
\begin{align*}
 \zeta_0:[0,1]\to D',\qquad
 \zeta_0(s)=s y_3 + (1-s) y_2,
\end{align*}
see Figure~\ref{f:D}.
\begin{figure}
\begin{footnotesize}
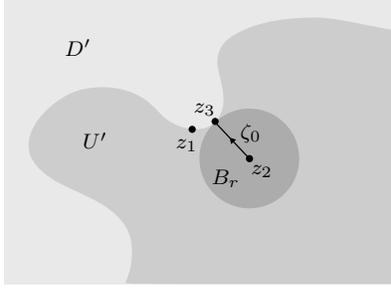
\end{footnotesize}
\caption{The open disk $D'$.}
\label{f:D}
\end{figure}
The function $\tau\circ\zeta_0|_{[0,1)}$ must be unbounded, since $\zeta_0(1)\not\in U'$. Therefore, Proposition~\ref{p:unbounded_Ws} implies that $\zeta_0(1)\in\Ws(\Kbr)$, where $\Kbr\subset N$ is the broken binding of the broken book decomposition. The stable lamination $\Ws(\Kbr)$ is tangent to the boundary $\partial B_r$ at $\zeta_0(1)$; indeed, any sufficiently short smooth curve $\ell\subset \Ws(\Kbr)$ such that $\zeta_0(1)\in\ell$ must be disjoint from the open set $U$ where the arrival time $\tau$ is well-defined. In particular, we have \[\dot\zeta_0(1)\pitchfork \Ws(\Kbr).\] 
Proposition~\ref{p:large_return} implies that the curve $\zeta_1:=\psi\circ\zeta_0:[0,1)\to D'$ has diameter
\begin{align*}
\diam(\zeta_1([0,1)))\geq c(\Sigma,\Sigma).
\end{align*}
However, this implies
\[\diam(D)\geq\diam(D')\geq\diam(\zeta_1([0,1))\geq c(\Sigma,\Sigma),\] 
which contradicts property (a) above, and thus shows that~\eqref{e:added_first_subcase_D} does not hold.

According to~\eqref{e:tau_n_unbounded}, it remains to consider the case $D'\setminus\psi(U')\neq\varnothing$. The argument is analogous to the one of the previous paragraph, replacing the stable lamination $\Ws(\Kbr)$ with the unstable lamination $\Wu(\Kbr)$. Briefly, if we set 
\[\tau_-:=-\tau\circ\psi^{-1}:\psi(U)\to(-\infty,0),\]
we can find a smooth path $\zeta_0:[0,1]\to D'$ such that $\zeta_0([0,1))\subset\psi(U')$, the function  $\tau_-\circ\zeta_0|_{[0,1)}$ is unbounded, $\zeta_0(1)\in\Wu(\Kbr)$, and $\dot\zeta_0(1)\pitchfork\Wu(\Kbr)$.
Proposition~\ref{p:large_return} and Remark~\ref{r:lambda_-lambda} imply 
\[\diam(D)\geq\diam(D')\geq\diam(\psi^{-1}\circ\zeta_0([0,1))\geq c(\Sigma,\Sigma),\] 
which once again contradicts property (a) above.
\end{proof}

\subsection{Geodesic flows}
\label{ss:proof_geodesic_flows}
Let $(M,g)$ be a closed Riemannian surface. Its unit tangent bundle
$S^gM:=\big\{ (x,v)\in TM\ \big|\ \|v\|_g=1 \big\}$
is equipped with the canonical contact form
\begin{align*}
\lambda_{(x,v)}(w)=g(v,d\pi(x,v)w),\qquad
\forall(x,v)\in S^gM,\ w\in T_{(x,v)}S^gM.
\end{align*}
Here, $\pi:S^gM\to M$ denotes the base projection $\pi(x,v)=x$. The associated Reeb vector field $X^g$, defined by the equations $\lambda(X^g)\equiv1$ and $d\lambda(X^g,\cdot)\equiv0$, is the geodesic vector field. The associated Reeb flow $\phi^g_t:S^gM\to S^gM$ is the geodesic flow;  namely, its orbits are of the form $\phi^g_t(\gamma(0))=(\gamma(t))$, where the curve $\gamma=(x,\dot x)$ is the lift of a geodesic $x:\R\to M$ of $(M,g)$ parametrized with unit speed $\|\dot x\|_g\equiv1$. From now on, with a slight abuse of terminology, we will call $g$-geodesics the orbits of $\phi^g_t$, and closed $g$-geodesics the closed orbits of $\phi^g_t$. In classical Riemannian terminology \cite{Anosov:1982ti}, a Riemannian metric all of whose closed geodesics are non-degenerate (in the sense of Equation~\ref{e:strong_nondegeneracy}) is called \emph{bumpy}.

Let $\gamma$ be a closed $g$-geodesic, and $\Sigma\subset S^gM$ a surface transverse to the geodesic vector field $X^g$ and containing the point $z_0:=\gamma(0)$. Let $t_0>0$ be the minimal period of $\gamma$. On a sufficiently small open neighborhood $U\subset\Sigma$ of $z_0$ there exists a unique smooth function $\tau:U\to(0,\infty)$ such that $\tau(z_0)=t_0$ and
\begin{align}
\label{e:return_map}
\psi(z):=\phi^g_{\tau(z)}(z)\in\Sigma,\qquad\forall z\in U.
\end{align} 
As we already remarked at the end of Section~\ref{ss:broken_book_decomposition}, the first-return map is a diffeomorphism onto its image 
$\psi:U\to\Sigma$
that preserves the area form $d\lambda|_\Sigma$, i.e.\ $\psi^*d\lambda=d\lambda$. Notice that $z_0$ is a fixed point of $\psi$, and the Floquet multipliers of $\gamma$ are precisely the eigenvalues of the linearized first return map $d\psi(z_0)$.

We denote by $\GG^2(
M)$ the space of smooth Riemannian metrics on the closed surface $M$, equipped with the $C^2$ topology. We then denote by 
\[\AAA(M)\subset\GG^2(M)\]  
the subset of those $g\in\GG^2(M)$ whose geodesic flow $\phi^g_t$ is Anosov. It is well known that $\AAA(M)$ is open in $\GG^2(M)$, see \cite{Anosov:1967wm}.

\begin{Lemma}
\label{l:stably_no_elliptic}
Let $(M,g_0)$ be a closed Riemannian surface such that every Riemannian metric $g_1$ sufficiently $C^2$-close to $g_0$ has no elliptic closed $g_1$-geodesics. Then $g_0\in\overline{\AAA(M)}$.
\end{Lemma}

\begin{proof}
By Kupka-Smale's theorem, which was proved for geodesic flows by Contreras and Paternain \cite[Theorem~2.5]{Contreras:2002vb}, there exists a residual subset $\BB\subset\GG^2(M)$ such that every $g\in\BB$ is bumpy and each pair of (not necessarily distinct) hyperbolic closed $g$-geodesics $\gamma_1,\gamma_2$ have transverse stable and unstable manifolds $\Ws(\gamma_1)\pitchfork\Wu(\gamma_2)$.

Let us assume that a Riemannian metric $g_0\in\GG^2(M)$ has an open neighborhood $\VV\subset\GG^2(M)$ such that no Riemannian metric $g\in\VV$ has an elliptic closed $g$-geodesic. A priori, this only implies that, for every $g\in\VV$, each closed $g$-geodesic  is hyperbolic. A stronger result due to Contreras and Paternain \cite[Theorem~D]{Contreras:2002vb} implies that actually, for every $g\in\VV$,  the closure of the space of closed $g$-geodesics $\overline{\Per(X^{g})}$ is a hyperbolic invariant set of the geodesic flow $\phi^g_t$. This allows us to apply Theorem~\ref{t:Reeb}, which implies that the geodesic flow $\phi^g_t$ of every $g\in\VV\cap\BB$ is Anosov, i.e.
\begin{align*}
\VV\cap\BB\subset\AAA(M).
\end{align*}
Therefore
$\VV\subset\overline{\VV\cap\BB}\subset\overline{\AAA(M)}$.
\end{proof}

We denote by
\begin{align*}
 \EE(M)\subset\GG^2(M)
\end{align*}
the open subset of those Riemannian metrics $g$ on the closed surface $M$ having a closed $g$-geodesic with Floquet multipliers in $S^1\setminus\{1,-1\}$. 

\begin{Lemma}
\label{l:elliptic}
Let $(M,g)$ be a closed Riemannian surface having an elliptic closed geodesic. Then $g\in\overline{\EE(M)}$.
\end{Lemma}

\begin{proof}
Let $(M,g_0)$ be a closed Riemannian surface with an elliptic closed $g_0$-geodesic $\gamma$. We denote by $\sigma,\sigma^{-1}\in S^1$ the Floquet multipliers of $\gamma$. If $\sigma\in S^1\setminus\{1,-1\}$, then $g_0\in\EE(M)$. Assume now that $\sigma\in\{1,-1\}$. Let $\Sigma\subset S^{g_0}M$ be an embedded surface transverse to the geodesic vector field $X^{g_0}$ and containing $z_0:=\gamma(0)$. On a sufficiently small open neighborhood $U\subset\Sigma$ of $z_0$ we have a smooth first-return map $\psi:U\to\Sigma$ defined as in~\eqref{e:return_map}. At its fixed point $z_0=\psi(z_0)$, we consider the linearized first-return map $d\psi(z_0)$, which is a linear symplectic isomorphism of the symplectic vector space $(T_{z_0}\Sigma,d\lambda)$. We identify this latter vector space with the standard symplectic vector space $(\C,\omega_0)$, where $\omega_0=\tfrac i2 d\overline z\wedge dz$,  so that $d\psi(z_0)$ is represented by a matrix $P_0\in\Sp(2)=\mathrm{SL}(2,\R)$.  
We denote by $W\subset\Sp(2)$ the open subset of those symplectic matrices whose eigenvalues are in $S^1\setminus\{1,-1\}$. Since the eigenvalues of $P_0$ are both equal to $\sigma\in\{1,-1\}$, we have
\begin{align*}
 P_0\in \overline W.
\end{align*}
We now apply a version of the Franks lemma for geodesic flows of surfaces, which was established by Contreras and Paternain \cite[Corollary~4.2]{Contreras:2002vb}: for every open neighborhood $\UU\subset\GG^2(M)$ of $g_0$, there exists an open neighborhood $Z\subset\Sp(2)$ of $P_0$ such that, for each $P_1\in Z$, 
there exists a Riemannian metric $g_1\in\UU$  such that $\gamma$ is also a closed $g_1$-geodesic and its linearized first-return map with respect to $g_1$ can be identified with $P_1$. By choosing $P_1\in Z\cap W$ we infer that, as a closed $g_1$-geodesic, $\gamma$ has Floquet multipliers in $S^1\setminus\{1,-1\}$, and therefore $g_1\in\EE(M)$. Since the neighborhood $\UU$ was arbitrary, this proves that $g_0\in\overline{\EE(M)}$. 
\end{proof}

\begin{proof}[Proof of Theorem~\ref{t:main}]
By Lemmas~\ref{l:stably_no_elliptic} and~\ref{l:elliptic}, the open subset $\AAA(M)\cup\EE(M)$ is dense in $\GG^2(M)$.
\end{proof}

\begin{proof}[Proof of Corollary~\ref{c:genus_0_1}]
By a result of Margulis \cite[Appendix]{Anosov:1967uq} (see also \cite{Plante:1972wn}), 
if the geodesic flow of a closed Riemannian surface $(M,g)$ is Anosov, then the fundamental group $\pi_1(S^gM)$ has exponential growth, and therefore the surface $M$ has genus larger than one. This, together with Theorem~\ref{t:main}, implies the corollary. 
\end{proof}

Let $(\Sigma,\omega)$ be a 2-dimensional symplectic manifold, that is, a surface $\Sigma$ equipped with an area 2-form $\omega$. We consider the space of symplectomorphisms 
\begin{align*}
\Symp(\Sigma,\omega) = \Big\{ \phi\in\Diff(\Sigma)\ \Big|\ \phi^*\omega=\omega \Big\}.
\end{align*}
Here, $\Diff(\Sigma)$ denotes as usual the space of smooth diffeomorphisms of $\Sigma$. Consider a symplectomorphism $\phi\in\Symp(\Sigma,\omega)$ having a fixed point $z\in\Fix(\phi)$. With the analogous terminology already introduced in Section~\ref{ss:non_degenerate} for Reeb orbits, the Floquet multipliers $\sigma,\sigma^{-1}$ of $z$ are the eigenvalues of $d\phi(z)$. The fixed point $z$ is:
\begin{itemize}

\item \emph{elliptic} when $\sigma\in S^1$,

\item  \emph{hyperbolic} when $\sigma\in\R\setminus\{-1,0,1\}$,

\item $k_0$-\emph{elementary}, for an integer $k_0\geq1$, when $\sigma^k\neq1$ for all $k\in\{1,...,k_0\}$.

\end{itemize}
Notice that a $k_0$-elementary fixed point can be either elliptic or hyperbolic.

We now focus on the special case of the 2-dimensional symplectic vector space $(\C,\omega_0)$, whose symplectic form is given by $\omega_0=\tfrac i2 d\overline z\wedge dz$. We consider the subspace of symplectomorphisms  fixing the origin
\begin{align*}
 \Symp_*(\C,\omega_0) = \Big\{ \phi\in\Symp(\C,\omega_0)\ \Big|\ \phi(0)=0 \Big\}.
\end{align*}
For each $k\geq1$, we denote by $\JJ^k$ the space of $k$-th jets at the origin $J^k_0\phi$ of such symplectic diffeomorphisms $\phi\in\Symp_*(\C,\omega_0)$. Notice that $\JJ^1$ is simply the linear symplectic group $\Sp(2)=\mathrm{SL}(2,\R)$.

Let us recall some classical notions that go back to Birkhoff \cite{Birkhoff:1920tl}. Let $\phi\in\Symp_*(\C,\omega_0)$ be a symplectomorphism with a 4-elementary elliptic fixed point at the origin. It turns out that, up to conjugation with a suitable symplectomorphism on a neighborhood $U\subset\C$ of the origin,  $\phi|_U$ can be put in Birkhoff normal form
\begin{align*}
 \phi|_U(z)
 =
 z\,e^{i(\theta + \beta|z|^2)} + \widetilde\phi(z),
\end{align*}
where $\theta\in\R$, $\beta\in\R$, and $\widetilde\phi$ is a smooth map that vanishes to order 4 at the origin. Notice that $e^{i\theta},e^{-i\theta}\in S^1$ are the Floquet multipliers of the fixed point at the origin, and therefore they are not $k$-th root of unity for $k=1,2,3,4$. The fixed point $0$ is called \emph{stable elliptic} when $\beta\neq0$. This is a condition on the 4-th jet $J^4_0\phi\in\JJ^4$. More generally, if $(\Sigma,\omega)$ is a 2-dimensional symplectic manifold, a fixed point $z_0$ of a symplectomorphism $\phi\in\Symp(\Sigma,\omega)$ is called stable elliptic when there is a chart $\kappa$ defined on a neighborhood of $z_0$ such that $\kappa(z_0)=0$, $\kappa^*\omega_0=\omega$, and  $\kappa\circ\phi\circ\kappa^{-1}$ is in Birkhoff normal form with a stable elliptic fixed point at the origin.

We denote by
\begin{align*}
 \SSS\subset\JJ^4
\end{align*}
the union of the conjugacy classes of the 4-th jets $J^4_0\phi$ of those symplectic diffeomorphisms $\phi\in\Symp_*(\C,\omega_0)$ in Birkhoff normal form with a stable elliptic fixed point at the origin.
We denote by
\begin{align*}
 \HH\subset\JJ^4
\end{align*}
the space of 4-th jets $J^4_0\phi$ of those symplectic diffeomorphisms $\phi\in\Symp_*(\C,\omega_0)$ with a hyperbolic fixed point at the origin.
The union $\SSS\cup\HH$ is open and dense in $\JJ^4$. 
Finally, for each $\sigma\in \R\cup S^1\subset\C$, we denote by 
\begin{align*}
\FF_\sigma\subset\JJ^4
\end{align*}
the open and dense subspace of 4-th jets $J^4_0\phi$ of those symplectic diffeomorphisms $\phi\in\Symp_*(\C,\omega_0)$ whose fixed point at the origin is 4-elementary and has Floquet multipliers different from $\sigma,\sigma^{-1}$.

We now apply these notions to the geodesic flows.
Let $(M,g)$ be a closed Riemannian surface, $\gamma$ a closed $g$-geodesic, and $\Sigma\subset S^gM$ a surface transverse to the geodesic vector field $X^g$ and containing the point $z_0:=\gamma(0)$. 
On a sufficiently small open neighborhood $U\subset\Sigma$ of $z_0$ we have a smooth first-return map $\psi:U\to\Sigma$ defined as in~\eqref{e:return_map}, which preserves the area form $d\lambda|_\Sigma$, i.e.\ $\psi^*d\lambda=d\lambda$, where $\lambda$ is the canonical contact form of $S^gM$.
 Notice that $z_0$ is a fixed point of $\psi$, and it is hyperbolic or elliptic if the closed geodesic $\gamma$ is hyperbolic or elliptic respectively. Extending this terminology, we say that $\gamma$ is $k_0$-elementary if the fixed point $z_0$ is $k_0$-elementary.
On a sufficiently small open neighborhood $V\subset U$ of $z_0$, there exists a Darboux chart, which is a diffeomorphism onto its image $\kappa:V\to\C$ such that $\kappa(z_0)=0$, $\kappa^*\omega_0=d\lambda$. The conjugated map $\kappa\circ\psi\circ\kappa^{-1}$, which is well defined on a neighborhood of the origin, has a fixed point at the origin. We denote by 
\begin{align*}
 \CC_\gamma\subset\JJ^4
\end{align*}
the conjugacy class of the 4-th jet $J^4_0(\kappa\circ\psi\circ\kappa^{-1})\in\JJ^4$. Notice that $\CC_\gamma$ is independent of the choice of the Darboux chart $\kappa$.
We say that the closed geodesic $\gamma$ is  \emph{stable elliptic} when $\CC_\gamma\subset \SSS$, that is, when there exists a Darboux chart $\kappa$ as above such that $\kappa\circ\psi\circ\kappa^{-1}$ is in Birkhoff normal form with a stable elliptic fixed point at the origin.

After these preliminaries, we address the $C^2$-stability conjecture for geodesic flows of closed Riemannian surfaces (Theorem~\ref{t:structural_stability}). By definition, if the geodesic flow $\phi_t^{g_0}$ of some Riemannian metric $g_0\in\GG^2(M)$ is $C^2$-structurally stable, then $g_0$ has  an open neighborhood 
\begin{align}
\label{e:nbhd_str_stab}
 \VV_{g_0}\subset\GG^2(M)
\end{align}
such that, for each $g_1\in\GG^2(M)$, there exists a homeomorphism
\begin{align*}
 \kappa_{g_1,g_0}:S^{g_0}M\to S^{g_1} M
\end{align*}
mapping $g_0$-geodesics to $g_1$-geodesics, i.e.
\begin{align*}
\Big\{\kappa_{g_1,g_0}\circ\phi_t^{g_0}(z)\ \Big|\ t\in\R\Big\} 
= 
\Big\{
\phi_t^{g_1}\circ\kappa_{g_1,g_0}(z)\ \Big|\ t\in\R
\Big\},\qquad \forall z\in S^{g_0}M.
\end{align*}
Notice that each $g_1\in\VV_{g_0}$ has a $C^2$-structurally stable geodesic flow  $\phi_t^{g_1}$ with associated neighborhood $\VV_{g_1}=\VV_{g_0}$.
The $C^2$-stability conjecture will be a direct consequence of Theorem~\ref{t:main} and of the following lemma.

\begin{Lemma}
\label{l:elliptic_prevents_stability}
If a closed Riemannian surface $(M,g)$ has an elliptic closed geodesic, then its geodesic flow $\phi^g_t$ is not $C^2$-struc\-tur\-al\-ly stable.
\end{Lemma}

\begin{proof}
This proof essentially follows from the work of Robinson \cite{Robinson:1970uo}.
Assume by contradiction that a Riemannian metric $g_0\in\GG^2(M)$ has an elliptic closed $g_0$-geodesic, but its geodesic flow $\phi_t^{g_0}$ is $C^2$-structurally stable. By Lemma~\ref{l:elliptic}, we have $g_0\in\overline{\EE(M)}$, and in particular $\EE(M)\cap\VV_{g_0}\neq\varnothing$. Notice that every $g_1\in\EE(M)\cap\VV_{g_0}$ has a $C^2$-structurally stable geodesic flow $\phi_t^{g_1}$ with $\VV_{g_1}=\VV_{g_0}$, and has a closed $g_1$-geodesic with Floquet multipliers in $S^1\setminus\{1,-1\}$. 

Since $\SSS\cup\HH$ is open and dense in $\JJ^4$, Kupka-Smale's theorem for geodesic flows \cite[Theorem~2.5]{Contreras:2002vb} implies that there exists a residual subset $\RR\subset\GG^2(M)$ such that, for each $g\in\RR$, every closed $g$-geodesic $\gamma$  satisfies $\CC_\gamma\subset\SSS\cup\HH$, that is, is hyperbolic or stable elliptic. Since $\RR$ is residual, the intersection $\RR\cap\EE(M)\cap\VV_{g_0}$ is non-empty, and we fix a Riemannian metric $g_1\in\RR\cap\EE(M)\cap\VV_{g_0}$, which must have a stable elliptic closed $g_1$-geodesic $\gamma_1$. 
We denote by $\sigma,\sigma^{-1}\in S^1\setminus\{1,-1\}$ the Floquet multipliers of $\gamma_1$, and apply once more Kupka-Smale's theorem for geodesic flows: since $\FF_\sigma$ is open and dense in $\JJ^4$, there exists a residual subset $\RR_\sigma\subset\GG^2(M)$ such that, for each $g\in\RR_\sigma$, every closed $g$-geodesic $\gamma$ satisfies $\CC_\gamma\subset\FF_\sigma$, that is, $\gamma$ is 4-elementary and has Floquet multipliers different from $\sigma,\sigma^{-1}$. Since $\RR_\sigma$ is residual, the intersection $\RR_\sigma\cap\VV_{g_0}$ is non-empty, and we fix a Riemannian metric therein $g_2\in\RR_\sigma\cap\VV_{g_0}$. Since both $g_1$ and $g_2$ are in $\VV_{g_0}$, we have a homeomorphism 
\[\kappa=\kappa_{g_2,g_1}:S^{g_1}M\to S^{g_2}M\] mapping $g_1$-geodesics to $g_2$-geodesics, i.e.
\begin{align*}
\Big\{\kappa\circ\phi_t^{g_1}(z)  \ \Big|\ t\in\R \Big\}
= 
\Big\{\phi_t^{g_2}\circ\kappa(z)  \ \Big|\  t\in\R\Big\},
\qquad
\forall z\in S^{g_1}M.
\end{align*}
We denote by $\gamma_2:=\kappa(\gamma_1)$ the closed $g_2$-geodesic corresponding to $\gamma_1$.
We fix a point $z_1\in\gamma_1$ and its image $z_2:=\kappa(z_1)\in\gamma_2$. Let $\Sigma_i\subset S^{g_i}M$ be an embedded surface transverse to the geodesic vector field $X^{g_i}$ and containing a point $z_i\in\gamma_i$. Let $U_i\subset\Sigma_i$ be a small enough open neighborhood of $z_i$ over which we have a well defined first-return map $\psi_i:U_i\to\Sigma_i$ as in~\eqref{e:return_map}. The point $z_1$ is a stable elliptic fixed point of $\psi_1$ with Floquet multipliers $\sigma,\sigma^{-1}\in S^1\setminus\{1,-1\}$, whereas every fixed point of $\psi_2$ is 4-elementary but with Floquet multipliers different than $\sigma,\sigma^{-1}$.

Notice that if $\kappa$ were a smooth diffeomorphism we could have chosen $\Sigma_2=\kappa(\Sigma_1)$; however, $\kappa$ is in general only a homeomorphism, and therefore the image $\kappa(\Sigma_1)$ is not necessarily smooth, let alone transverse to the geodesic vector field $X^{g_2}$. Nevertheless, for $\epsilon>0$ small enough, we have an embedded flow-box
\begin{align*}
 N:=\bigcup_{t\in(-\epsilon,\epsilon)} \phi_t^{g_2}(\Sigma_2),
\end{align*}
with associated smooth retraction $r:N\to\Sigma_2$ defined by
\begin{align*}
 r(\phi_t(z))=z,\qquad\forall z\in\Sigma_2,\ t\in(-\epsilon,\epsilon).
\end{align*}
A sufficiently small open neighborhood $V\subset U_1\cap \psi_1(U_1)$ of $z_1$ satisfies $\kappa(V)\subset N$.
The composition 
\[\chi:=r\circ\kappa:V\to \Sigma_2\] 
is a homeomorphism onto its image. On a sufficiently small open neighborhood $W\subset V$ of $z_1$, the map $\chi$ conjugates the first-return maps $\psi_1$ and $\psi_2$, i.e.
\begin{align*}
 \psi_1|_{W}=\chi^{-1}\circ\psi_2\circ\chi|_{W}.
\end{align*}
Since $z_1\in\Fix(\psi_1)$ is stable elliptic, and every $z\in\Fix(\psi_2)$ is 4-elementary, a statement due to Robinson \cite[Lemma~22]{Robinson:1970uo} implies that the corresponding fixed points $z_1$ and $z_2=\kappa(z_1)=\chi(z_1)$ have the same Floquet multipliers $\sigma,\sigma^{-1}$. This contradicts the fact that no fixed point of $\psi_2$ has Floquet multipliers $\sigma,\sigma^{-1}$.
\end{proof}

\begin{proof}[Proof of Theorem~\ref{t:structural_stability}]
Let $g_0\in\GG^2(M)$ be a Riemannian metric whose geodesic flow $\phi_t^{g_0}$ is $C^2$-structurally stable. Let $\VV_{g_0}\subset\GG^2(M)$ be the neighborhood of $g_0$ on which the structural stability holds, as defined in~\eqref{e:nbhd_str_stab}. Lemma~\ref{l:elliptic_prevents_stability} implies that every $g\in\VV_{g_0}$ has no elliptic closed $g$-geodesics, that is, all the closed $g$-geodesics are hyperbolic. By Theorem~\ref{t:main}, there exists a dense open subset $\UU\subset\GG^2(M)$ such that, for each $g\in\UU$, either there exists an elliptic closed $g$-geodesic or the geodesic flow $\phi_t^{g}$ is Anosov. Therefore, every $g\in\VV_{g_0}\cap\UU$ has an Anosov geodesic flow $\phi_t^{g}$. 

The proof is not complete yet, since a priori the original Riemannian metric $g_0$ might not be contained in $\UU$. Since for every $g\in\VV_{g_0}$ all the closed $g$-geodesics are hyperbolic, a theorem due to Contreras and Paternain \cite[Theorem~D]{Contreras:2002vb} implies that the closure of the space of closed $g_0$-geodesics $\overline{\Per(X^{g_0})}$ is a hyperbolic invariant set of the geodesic flow $\phi^{g_0}_t$. Consider a Riemannian metric $g_1\in\VV_{g_0}\cap\UU$, and the associated homeomorphism $\kappa_{g_1,g_0}:S^{g_0}M\to S^{g_1}M$ mapping $g_0$-geodesics to $g_1$-geodesics. Since the geodesic flow $\phi_t^{g_1}$ is Anosov, the space of closed $g_1$-geodesics $\Per(X^{g_1})$ is dense in $S^{g_1}M$, i.e. \[\overline{\Per(X^{g_1})}=S^{g_1}M.\]
Therefore
\begin{align*}
\overline{\Per(X^{g_0})} 
= \overline{\kappa_{g_1,g_0}^{-1}(\Per(X^{g_0}))}
=
\kappa_{g_1,g_0}^{-1}(\overline{\Per(X^{g_0})})
=
\kappa_{g_1,g_0}^{-1}(S^{g_1}M)
=
S^{g_0}M.
\end{align*}
This, together with the hyperbolicity of $\overline{\Per(X^{g_0})}$, means that the geodesic flow $\phi_t^{g_0}$ is Anosov.\end{proof}

\begin{proof}[Proof of Theorem~\ref{t:low_genus}]
Let $M$ be a closed surface of genus at most one, equipped with a Riemannian metric $g_0$. By Corollary~\ref{c:genus_0_1}, there exists a smooth Riemannian metric $g$ on $M$ that is arbitrarily $C^2$-close to $g_0$ and has an elliptic closed $g$-geodesic. This, together with Lemma~\ref{l:elliptic_prevents_stability}, implies that the geodesic flow $\phi_t^{g_0}$ is not $C^2$-structurally stable.
\end{proof}

\bibliographystyle{amsalpha}
\bibliography{biblio}

\def\cprime{$'$} \def\cprime{$'$} \def\cprime{$'$} \def\cprime{$'$}
\providecommand{\bysame}{\leavevmode\hbox to3em{\hrulefill}\thinspace}
\providecommand{\MR}{\relax\ifhmode\unskip\space\fi MR }
\providecommand{\MRhref}[2]{%
  \href{http://www.ams.org/mathscinet-getitem?mr=#1}{#2}
}
\providecommand{\href}[2]{#2}
\begin{thebibliography}{CGHHL21}

\bibitem[Ano67]{Anosov:1967wm}
D.~V. Anosov, \emph{Geodesic flows on closed {R}iemannian manifolds of negative
  curvature}, Trudy Mat. Inst. Steklov. \textbf{90} (1967), 209.

\bibitem[Ano82]{Anosov:1982ti}
\bysame, \emph{Generic properties of closed geodesics}, Izv. Akad. Nauk SSSR
  Ser. Mat. \textbf{46} (1982), no.~4, 675--709, 896.

\bibitem[Arn98]{Arnaud:1998aa}
M.-C. Arnaud, \emph{Le ``closing lemma'' en topologie {$C^1$}}, M\'{e}m. Soc.
  Math. Fr. (N.S.) (1998), no.~74, vi+120.

\bibitem[AS67]{Anosov:1967uq}
D.~V. Anosov and Ja.~G. Sina\u{\i}, \emph{Certain smooth ergodic systems},
  Russian Math. Surveys \textbf{22} (1967), no.~5, 103--167.

\bibitem[Ban93]{Bangert:1993ue}
V.~Bangert, \emph{On the existence of closed geodesics on two-spheres},
  Internat. J. Math. \textbf{4} (1993), no.~1, 1--10.

\bibitem[Bir20]{Birkhoff:1920tl}
G.~D. Birkhoff, \emph{Surface transformations and their dynamical
  applications}, Acta Math. \textbf{43} (1920), 1--119.

\bibitem[BL98]{Bonatti:1998uo}
C.~Bonatti and R.~Langevin, \emph{Diff\'eomorphismes de {S}male des surfaces},
  Ast\'erisque (1998), no.~250, viii+235.

\bibitem[Bow72]{Bowen:1972ws}
R.~Bowen, \emph{Periodic orbits for hyperbolic flows}, Amer. J. Math.
  \textbf{94} (1972), 1--30.

\bibitem[BR75]{Bowen:1975ua}
R.~Bowen and D.~Ruelle, \emph{The ergodic theory of {A}xiom {A} flows}, Invent.
  Math. \textbf{29} (1975), no.~3, 181--202.

\bibitem[CBP02]{Contreras:2002vb}
G.~Contreras-Barandiar\'{a}n and G.~P. Paternain, \emph{Genericity of geodesic
  flows with positive topological entropy on {$S^2$}}, J. Differential Geom.
  \textbf{61} (2002), no.~1, 1--49.

\bibitem[CDHR22]{Colin:2022aa}
V.~Colin, P.~Dehornoy, U.~Hryniewicz, and A.~Rechtman, \emph{Generic properties
  of $3$-dimensional {R}eeb flows: {B}irkhoff sections and entropy},
  arXiv:2202.01506, 2022.

\bibitem[CDR20]{Colin:2020tl}
V.~Colin, P.~Dehornoy, and A.~Rechtman, \emph{On the existence of supporting
  broken book decompositions for contact forms in dimension 3}, to appear in
  Invent. Math., 2020.

\bibitem[CGHHL21]{Cristofaro-Gardiner:2021tc}
D.~Cristofaro-Gardiner, U.~Hryniewicz, M.~Hutchings, and H.~Liu, \emph{Contact
  three-manifolds with exactly two simple {R}eeb orbits}, arXiv:2102.04970,
  2021.

\bibitem[CGHP19]{Cristofaro-Gardiner:2019wf}
D.~Cristofaro-Gardiner, M.~Hutchings, and D.~Pomerleano, \emph{Torsion contact
  forms in three dimensions have two or infinitely many {R}eeb orbits}, Geom.
  Topol. \textbf{23} (2019), no.~7, 3601--3645.

\bibitem[CKW21]{Climenhaga:2021aa}
V.~Climenhaga, G.~Knieper, and K.~War, \emph{Uniqueness of the measure of
  maximal entropy for geodesic flows on certain manifolds without conjugate
  points}, Adv. Math. \textbf{376} (2021), Paper No. 107452, 44.

\bibitem[CM22]{Contreras:2022aa}
G.~Contreras and M.~Mazzucchelli, \emph{Existence of {B}irkhoff sections for
  {K}upka-{S}male {R}eeb flows of closed contact 3-manifolds}, Geom. Funct.
  Anal. \textbf{32} (2022), no.~5, 951--979.

\bibitem[CO04]{Contreras:2004vh}
G.~Contreras and F.~Oliveira, \emph{{$C^2$}-densely, the 2-sphere has an
  elliptic closed geodesic}, Ergodic Theory Dynam. Systems \textbf{24} (2004),
  no.~5, 1395--1423, Michel Herman's memorial issue.

\bibitem[Con14]{Contreras:2014vo}
G.~Contreras, \emph{Generic {M}a\~n\'e {S}ets}, arXiv:1410.7141, 2014.

\bibitem[FH19]{Fisher:2019vz}
T.~Fisher and B.~Hasselblatt, \emph{{H}yperbolic {F}lows}, Zurich Lectures in
  Advanced Mathematics, European Mathematical Society, Berlin, 2019.

\bibitem[Fra92]{Franks:1992wu}
J.~Franks, \emph{Geodesics on {$S^2$} and periodic points of annulus
  homeomorphisms}, Invent. Math. \textbf{108} (1992), no.~2, 403--418.

\bibitem[Grj79]{Grjuntal:1979uy}
A.~I. Grjuntal{\cprime}, \emph{The existence of convex spherical metrics all of
  whose closed nonselfintersecting geodesics are hyperbolic}, Izv. Akad. Nauk
  SSSR Ser. Mat. \textbf{43} (1979), no.~1, 3--18, 237.

\bibitem[Gro85]{Gromov:1985ww}
M.~Gromov, \emph{Pseudo holomorphic curves in symplectic manifolds}, Invent.
  Math. \textbf{82} (1985), no.~2, 307--347.

\bibitem[Hin93]{Hingston:1993wb}
N.~Hingston, \emph{On the growth of the number of closed geodesics on the
  two-sphere}, Internat. Math. Res. Notices (1993), no.~9, 253--262.

\bibitem[Hur82]{Hurley:1982uc}
M.~Hurley, \emph{Attractors: Persistence, and density of their basins}, Trans.
  Amer. Math. Soc. \textbf{269} (1982), no.~1, 247--271.

\bibitem[HWZ98]{Hofer:1998vy}
H.~Hofer, K.~Wysocki, and E.~Zehnder, \emph{The dynamics on three-dimensional
  strictly convex energy surfaces}, Ann. of Math. (2) \textbf{148} (1998),
  no.~1, 197--289.

\bibitem[HWZ02]{Hofer:2002vt}
\bysame, \emph{Pseudoholomorphic curves and dynamics in three dimensions},
  Handbook of dynamical systems, {V}ol. 1{A}, North-Holland, Amsterdam, 2002,
  pp.~1129--1188.

\bibitem[HWZ03]{Hofer:2003wf}
\bysame, \emph{Finite energy foliations of tight three-spheres and
  {H}amiltonian dynamics}, Ann. of Math. (2) \textbf{157} (2003), no.~1,
  125--255.

\bibitem[Iri15]{Irie:2015aa}
K.~Irie, \emph{Dense existence of periodic {R}eeb orbits and {ECH} spectral
  invariants}, J. Mod. Dyn. \textbf{9} (2015), 357--363.

\bibitem[Mn88]{Mane:1988vo}
R.~Ma\~n{\'e}, \emph{A proof of the ${C}\sp 1$ stability conjecture}, Inst.
  Hautes \'Etudes Sci. Publ. Math. (1988), no.~66, 161--210.

\bibitem[New77]{Newhouse:1977wn}
S.~E. Newhouse, \emph{Quasi-elliptic periodic points in conservative dynamical
  systems}, Amer. J. Math. \textbf{99} (1977), no.~5, 1061--1087.

\bibitem[Poi05]{Poincare:1905ub}
H.~Poincar{\'e}, \emph{Sur les lignes g\'eod\'esiques des surfaces convexes},
  Trans. Amer. Math. Soc. \textbf{6} (1905), no.~3, 237--274.

\bibitem[PR83]{Pugh:1983aa}
C.~C. Pugh and C.~Robinson, \emph{The {$C^{1}$} closing lemma, including
  {H}amiltonians}, Ergodic Theory Dynam. Systems \textbf{3} (1983), no.~2,
  261--313.

\bibitem[PS70]{Palis:1970tw}
J.~Palis and S.~Smale, \emph{Structural stability theorems}, Global {A}nalysis
  ({P}roc. {S}ympos. {P}ure {M}ath., {V}ol. {XIV}, {B}erkeley, {C}alif., 1968),
  Amer. Math. Soc., Providence, R.I., 1970, pp.~223--231.

\bibitem[PT72]{Plante:1972wn}
J.~F. Plante and W.~P. Thurston, \emph{Anosov flows and the fundamental group},
  Topology \textbf{11} (1972), 147--150.

\bibitem[Rif12]{Rifford:2012aa}
L.~Rifford, \emph{Closing geodesics in {$C^1$} topology}, J. Differential Geom.
  \textbf{91} (2012), no.~3, 361--381.

\bibitem[Rob70]{Robinson:1970uo}
R.~C. Robinson, \emph{Generic properties of conservative systems}, Amer. J.
  Math. \textbf{92} (1970), 562--603.

\bibitem[RR21]{Rifford:2021aa}
L.~Rifford and R.~Ruggiero, \emph{On the stability conjecture for geodesic
  flows of manifolds without conjugate points}, Ann. H. Lebesgue \textbf{4}
  (2021), 759--784.

\bibitem[Rug91]{Ruggiero:1991aa}
R.~O. Ruggiero, \emph{On the creation of conjugate points}, Math. Z.
  \textbf{208} (1991), no.~1, 41--55.

\bibitem[Rug21]{Ruggiero:2021aa}
\bysame, \emph{Does the hyperbolicity of periodic orbits of a geodesic flow
  without conjugate points imply the {A}nosov property?}, Trans. Amer. Math.
  Soc. \textbf{374} (2021), no.~10, 7263--7280.

\bibitem[Sch21]{Schulz:2021wk}
B.~H. Schulz, \emph{Characterization of {A}nosov metrics on closed surfaces and
  stable ergodicity}, Ph.D. thesis, Ruhr-Universit\"at Bochum, 2021.

\end{thebibliography}
\end{document}